\documentclass[11pt]{article}

\usepackage[english]{babel}

\usepackage{amsmath,amssymb,amsfonts,amsthm,mathrsfs,bbm}
\usepackage{graphicx,xcolor}
\usepackage{enumerate,geometry,marginnote,afterpage}
\usepackage{fancyhdr,courier,type1cm,dsfont,float,tikz}
\usepackage{hyperref}

\numberwithin{equation}{section}

\theoremstyle{plain} 

\newtheorem{theorem}{Theorem}[section]

\newtheorem{lemma}[theorem]{Lemma}

\newtheorem{definition}[theorem]{Definition}

\newtheorem{remark}[theorem]{Remark}
\newtheorem{example}[theorem]{Example}

\newcommand{\N}{\mathbb N}

\newcommand{\EE}{\mathbb E}
\newcommand{\PP}{\mathbb{P}}
\newcommand{\Tt}{\mathcal{T}}

\newcommand{\tX}{\tilde{X}}
\newcommand{\tk}{\tilde{k}}
\newcommand{\tl}{\tilde{l}}
\newcommand{\tp}{\tilde{p}}
\newcommand{\tchi}{{\tilde{\chi}}}
\newcommand{\tf}{{\tilde{f}}}

\newcommand{\ind}[1]{\mathbbm{1}_{\left\{#1\right\}}}

\makeatletter
\renewcommand{\title}[1]{\gdef\@title{\Huge #1}}
\makeatother

\begin{document}
	
	\title{The friendship paradox for trees}
	
	\author{\renewcommand{\thefootnote}{\arabic{footnote}}
		Rajat Subhra Hazra
		\footnotemark[1]
		\\
		\renewcommand{\thefootnote}{\arabic{footnote}}
		Frank den Hollander
		\footnotemark[1]
		\\
		\renewcommand{\thefootnote}{\arabic{footnote}}
		Nelly Litvak
		\footnotemark[2]
		\\
		\renewcommand{\thefootnote}{\arabic{footnote}}
		Azadeh Parvaneh
		\footnotemark[1]
	}
	
	\footnotetext[1]{
		Mathematical Institute, Leiden University, Einsteinweg 55, 2333 CC Leiden, The Netherlands.\\
		\emph{email}: \{r.s.hazra;\,denholla;\,s.a.parvaneh.ziabari\}@math.leidenuniv.nl
	}
	
	\footnotetext[2]{
		Department of Mathematics and Computer Science, Eindhoven University of Technology, 
		P.O.\ Box 513, 5600 MB Eindhoven, The Netherlands.\\
		\emph{email}: n.v.litvak@tue.nl
	}
	
	\maketitle
\begin{abstract}
We analyse the friendship paradox on finite and infinite trees. In particular, we monitor the vertices for which the friendship-bias is positive, neutral and negative, respectively. For an arbitrary finite tree, we show that the number of positive vertices is at least as large as the number of negative vertices, a property we refer to as significance, and derive a lower bound in terms of the branching points in the tree. For an infinite Galton-Watson tree, we compute the densities of the positive and the negative vertices and show that either may dominate the other, depending on the offspring distribution. We also compute the densities of the edges having two given types of vertices at their ends, and give conditions in terms of the offspring distribution under which these types are negatively correlated.
	
	\medskip\noindent
	{\it AMS} 2020 {\it subject classifications.}
	05C80, 
	60C05. 

	
	\medskip\noindent
	{\it Acknowledgment.} The work in this paper was supported by the Netherlands Organisation for Scientific Research (NWO) through Gravitation-grant NETWORKS-024.002.003. FdH and NL were supported by the National Science Foundation (NSF) under Grant No.\ DMS-1928930 while in residence at the Simons Laufer Mathematical Sciences Institute in Berkeley, California, USA during the Spring 2025 semester. AP received funding from the European Union's Horizon 2020 research and innovation programme under the Marie Sk\l odowska-Curie grant agreement Grant Agreement No 101034253. The authors thank Joshua Wolf (TU Munich) for pointing out two errors in an earlier version of the paper. 
	
\end{abstract}

\vspace{0.1cm}
\hfill\includegraphics[scale=0.1]{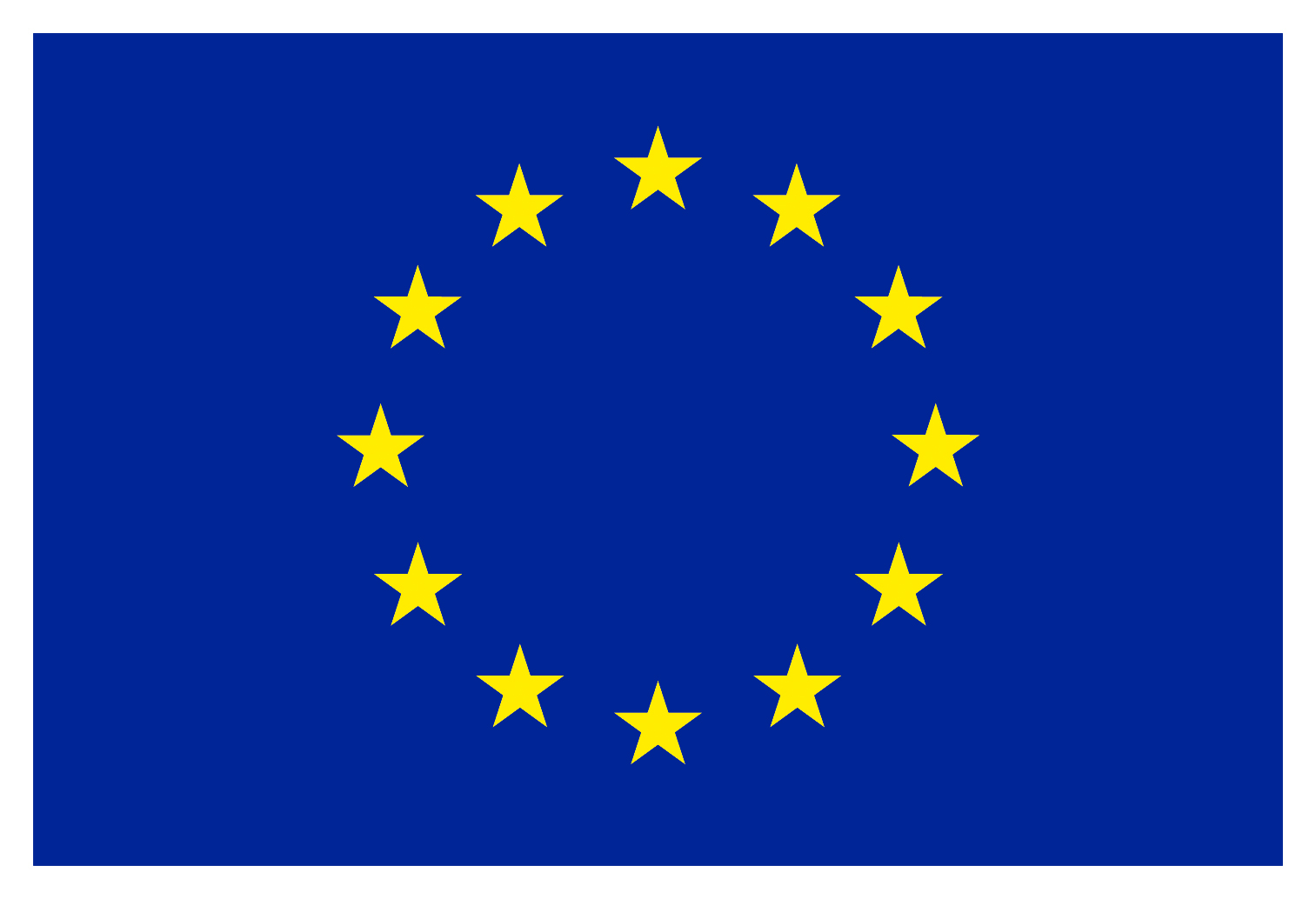}


\newpage 

\small
\tableofcontents
\normalsize


\section{Introduction and main results}

Section~\ref{sec:B} provides a brief background. Section~\ref{sec:AP-not} introduces notation. Section~\ref{sec:FP-finite-tree} states our main theorems for finite trees, Section~\ref{sec:FP-infinite-tree} for infinite trees. Section~\ref{sec:DiscPro} places these theorems in their proper context and lists a few open problems. Section~\ref{sec:outline} gives the outline of the remainder of the paper.


\subsection{Background}
\label{sec:B}

The friendship paradox says that for any finite graph the \emph{average friendship-bias} is non-negative \cite{SF,HHP1}. More precisely, let $G_n$ be a \emph{finite undirected graph} with $n$ vertices, labelled by $u \in [n] = \{1,\ldots,n\}$. Define the friendship-bias of vertex $u$ as
\[
\Delta_{u,n} = \left(\frac{1}{d_u} \sum_{v \in [n]} A_{uv} d_v - d_u\right) \,\mathbbm{1}_{\{d_u \neq 0\}}, 
\]
where $A=(A_{uv})_{u,v\in [n]}$ is the adjacency matrix of $G_n$, i.e., $A_{uv}$ is the number of edges between $u \neq v$ while $A_{uu}$ is twice the number of self-loops at $u$, and $d_u = \sum_{v \in [n]} A_{uv}$ is the degree of $u$. The index $n$ is used to emphasise that the friendship-bias is computed in the graph $G_n$ with $n$ vertices. Then the friendship paradox says that
\[
\frac{1}{n} \sum_{u \in [n]} \Delta_{u,n} \geq 0,
\] 
with equality if and only if all connected components of $G_n$ are regular. Typically, there are $u$ for which $\Delta_{u,n}$ is $>0$, $=0$, $<0$. Even though the average is non-negative, there are graphs for which the number of vertices $u$ with $\Delta_{u,n} < 0$ is larger than the number of vertices $u$ with $\Delta_{u,n} > 0$. One of the goals of the present paper is to analyse the number and the arrangement of such vertices in \emph{finite trees} and \emph{infinite trees}.

The friendship paradox was introduced in 1991 by the American sociologist Scott Feld~\cite{SF}. Since then, numerous studies have explored and applied this paradox across various domains, including social sciences, information theory and epidemiology. For instance, researchers have used the friendship paradox to design efficient sampling methods for early detection of contagious outbreaks~\cite{CF}, and to improve strategies for delivering information and monitoring its propagation~\cite{KN}, as well as for reducing the spread of contagion~\cite{rosenblatt}. The friendship paradox has applications in exploring the network, for example, quickly finding vertices of high degrees \cite{avrachenkov2014quick} and developing immunisation strategies~\cite{cohen2003immunization, britton2007vaccination}. 

Apart from being interesting in itself, the friendship paradox has useful generalisations and implications. For instance, the friendship-bias can be viewed as a centrality measure, akin to PageRank centrality \cite{PageRank} and degree centrality. Co-authorship networks of Physical Review journals and Google Scholar profiles reveal that on average the co-authors of a person have more collaborations, publications and citations than that person has themselves \cite{EomJo}. On Twitter, most users, on average, follow others who are more popular and more active, i.e., tend to have more followers and share more viral content than they do themselves \cite{HKK}. The friendship paradox has implications for individual biases in perception and thought contagion as well, because our social norms are influenced by our perceptions of others, and are strongly shaped by the people around us \cite{Jackson, Lerman}. 

Despite many practical examples, \emph{mathematically} the friendship paradox has received only modest attention, and a proper \emph{quantification} of friendship-biases in large-size networks is still lacking. In \cite{CR}, it is shown that a randomly chosen friend of a randomly chosen individual has stochastically more friends than that individual. In \cite{HHP1}, a quantification of the friendship paradox for sparse locally tree-like random graphs is carried out, and the notion of \emph{significance} of the friendship paradox is introduced. Specifically, the friendship paradox is said to be asymptotically significant for a sequence of finite random graphs $(G_n)_{n\in\N}$ labelled by the number of vertices $n$ if the proportion of vertices with non-negative friendship-bias converges in probability as $n\to\infty$ to a number $\geq \tfrac{1}{2}$. Interestingly, it is shown that the friendship paradox is significant for the homogeneous Erd\H{o}s-R\'enyi random graph, the inhomogeneous Erd\H{o}s-R\'enyi random graph, as well as for a well-known case of the configuration model and the preferential attachment model. 

In \cite{HHP2}, the \emph{multi-level} friendship paradox is analysed, which addresses the friendship paradox at higher levels of friendship. More specifically, by extending the notion of friendship-bias to a multi-level setting, using random walks on graphs and imposing additional constraints on the graph, it is shown that the $k$-level average friendship-bias is non-negative for each fixed $k$, regardless of whether $k$-level friends are explored via backtracking or non-backtracking random walks.
 
Recent works have considered generalisations of the friendship paradox to wedges and triangles \cite{BGHJS}, and to various notions of network centrality \cite{HV}. Closer to our work, \cite{Lee2026} empirically demonstrates that in a small artificial network, as well as in the well-known Zachary's Karate Club network, the fraction of vertices $u$ with $\Delta_{u,n}>0$ is smaller than $\tfrac12$. Moreover, \cite{Lerman2025,Lee2026} consider yet another form of friendship paradox, namely, when the degree of a vertex is smaller than the median degree of its neighbours.  While the fraction of vertices that experience this form of friendship paradox is larger than $\tfrac12$ in empirical networks in \cite{Lerman2025,Lee2026}, it can be smaller than $\tfrac12$ in general, as demonstrated by a toy example in \cite{Lee2026}.

In the present paper we develop a probabilistic framework for understanding and analysing the friendship paradox for trees. We isolate and quantify the local structure underlying the friendship paradox on trees, which has both deterministic and random features. By formulating the paradox in terms of vertex-level bias and type, we derive sharp results on the distribution and correlation of these types in tree environments. Our results yield a systematic framework for studying when and how the friendship paradox manifests itself in tree-like structures that arise naturally in real-world networks.


\subsection{Notations}
\label{sec:AP-not}

Let $(\mathcal{T}_n,\phi)$ be a \emph{rooted tree} with $n \in \mathbb{N} \cup \{\infty\}$ vertices, vertex set $V(\mathcal{T}_n)$, and root vertex $\phi$. Throughout the paper, all random trees are defined on a common probability space $(\Omega,\mathcal{F},\PP)$, and expectation and variance with respect to $\PP$ are denoted by $\EE$ and $\mathbb{V}{\rm ar}$, respectively. To label the vertices, use the Ulam-Harris method of labelling a tree: each vertex $u \in V(\mathcal{T}_n) \setminus \{\phi\}$ is labelled by a finite word $\phi\, u^1 \ldots u^k$, where $k \in \mathbb{N}$ and, for $u \in V(\mathcal{T}_n)$, the $j$th offspring of $u$ is
\[
u_{j} = uj = \left\{
\begin{array}{ll}
\phi\, j, &\text{if $u=\phi$},\\[0.2cm]
\phi\, u^{1}\ldots u^{k} j, &\text{if $u=\phi\, u^{1}\ldots u^{k}$}.
\end{array} 
\right. 
\]

For $u, v \in V(\mathcal{T}_n)$, write $u\sim v$ to indicate that vertex $u$ is adjacent to vertex $v$. The edge set of $\mathcal{T}_n$ is $E(\mathcal{T}_n) = \{\{u,v\}\colon\, u,v\in V(\mathcal{T}_n), u \sim v\}$. Let $d_u$ be the degree of a vertex $u$. Define the \textit{friendship-bias} of $u$ as
\[
\Delta_{u,n} = \frac{1}{d_u} \sum_{ {v \in V(\mathcal{T}_n)} \atop {v \sim u} } d_v - d_u,
\]
which represents the difference between the average degree of the neighbours of $u$ and the degree of $u$ itself, with the index $n$ referring to the friendship-bias being computed in the tree $\mathcal{T}_n$ with $n$ vertices. Throughout the paper we follow the natural convention of defining the empty sum as $0$, in particular, the friendship-bias of the root in a trivial tree is $0$. A vertex is said to be \textit{positive} when its friendship-bias is strictly positive, \textit{negative} when its friendship-bias is strictly negative, and \textit{neutral} otherwise. Define, for $n \in \N$,
\[
\begin{aligned}
N_n^+ &= \big|\big\{u \in V(\mathcal{T}_n)\colon\,\Delta_{u,n} > 0\big\}\big|,\\
N_n^- &= \big|\big\{u \in V(\mathcal{T}_n)\colon\,\Delta_{u,n} < 0\big\}\big|,\\
N_ n^0 &= n - N_n^+ - N_n^-.
\end{aligned}
\]


\subsection{Friendship paradox for finite trees}
\label{sec:FP-finite-tree}

The friendship paradox says that for an arbitrary finite graph $G_n$ the \emph{average friendship-bias} is non-negative \cite{SF,HHP1}, i.e.,
\[
\frac{1}{|V(G_n)|} \sum_{u \in V(G_n)} \Delta_{u,n} \geq 0,
\] 
where $V(G_n)$ denotes the vertex set of $G_n$. Equality holds if and only if all connected components of $G_n$  are regular. Still, there are graphs for which the number of negative vertices is larger than the number of positive vertices. This leads us to the following notion. 

\begin{definition}{\bf [Significance for finite trees]}
{\rm For a finite tree $\mathcal{T}_n$, the friendship paradox is said to be \textit{significant} when $N_n^+ \geq N_n^-$, and \textit{insignificant} otherwise. It is said to be \textit{strictly significant} when $N_n^+ > N_n^-$.} \hfill$\spadesuit$
\label{def:finite-significance}
\end{definition}

Our first theorem says that all finite trees are strictly significant, unless they consist of a single path of length $\neq 3$, and provides a lower bound for the gap between the numbers of positive and negative vertices.

\begin{theorem}{\bf [Finite trees are significant]}
\label{thm:finite-trees-significant}
For all finite trees $\mathcal{T}_n$,  the following statements hold:
\begin{itemize}
\item[(a)] 
If there is an $u\in V(\Tt_n)$ such that $d_u\ge 3$, then
\begin{equation}
\label{eq:difference-finite-tree}
N_n^+ - N_n^- \geq 2 + \sum_{u \in V(\Tt_n)} (d_u-3) \ind{d_u \geq 3} > 0.
\end{equation}
\item[(b)] 
If $\Tt_n$ is a path of size $n$, then
\[
N_n^+-N_n^-=\ind{n=3}.
\]
\end{itemize}
\end{theorem}


\subsection{Friendship paradox for infinite trees}
\label{sec:FP-infinite-tree}

Next, consider an infinite Galton-Watson tree $\mathcal{T}_\infty$ with offspring distribution $X$, and let $p_k=\PP(X=k)$, $k\in\N_0$. This model arises naturally as the local weak limit of large sparse random graphs, including the Erd\H{o}s–R\'enyi and configuration models. Our goal is to characterise the limiting distribution of vertex types, i.e., positive, neutral, or negative, especially when the tree is explored from the root. Unlike the deterministic case, the randomness of the degrees introduces fluctuations that affect the sign of the friendship paradox at different generations. Simulations of such trees for various Poisson offspring distributions are shown in Figure~\ref{fig:sim}, illustrating how the relative frequencies of vertex types change with the mean degree.

\begin{figure}[htbp]
\begin{center}
\includegraphics[width=0.49\linewidth]{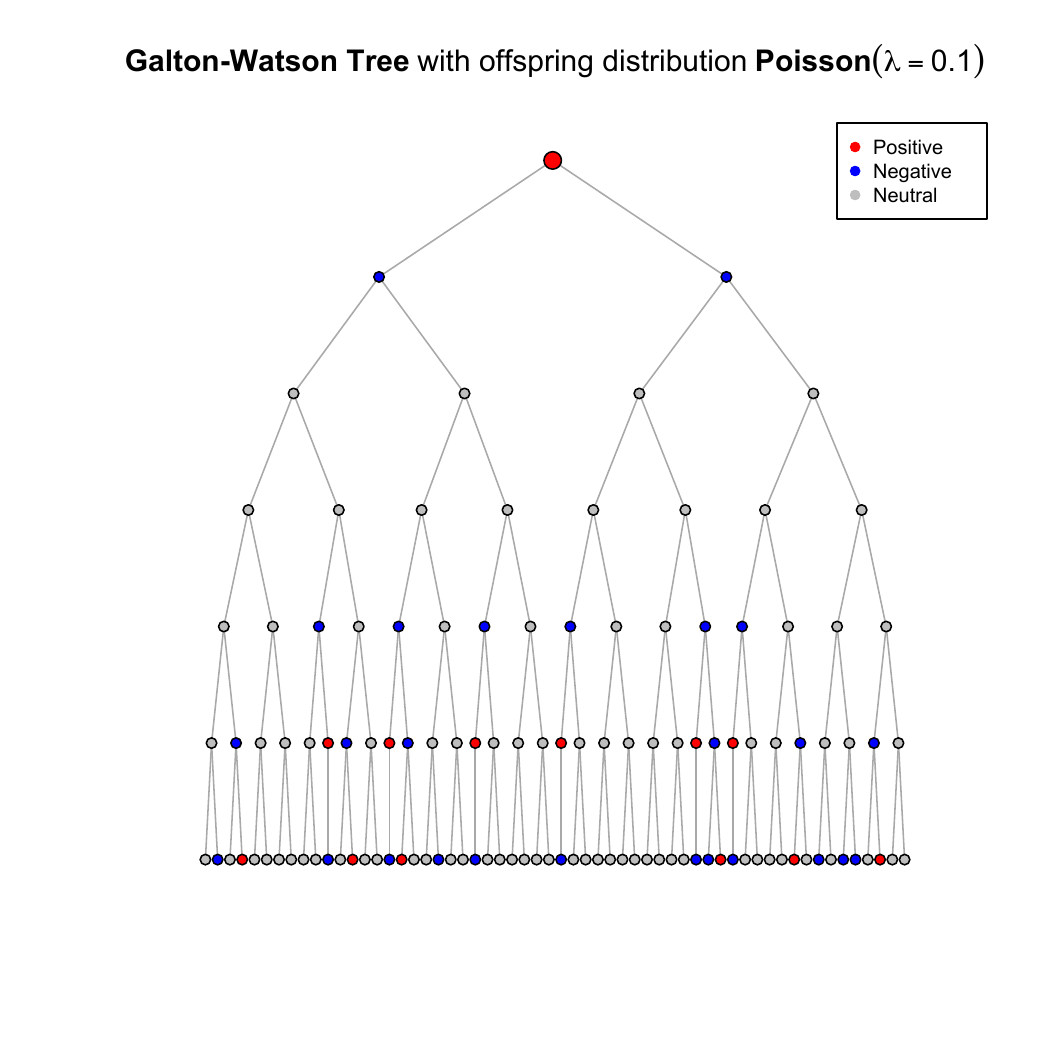}
\includegraphics[width=0.49\linewidth]{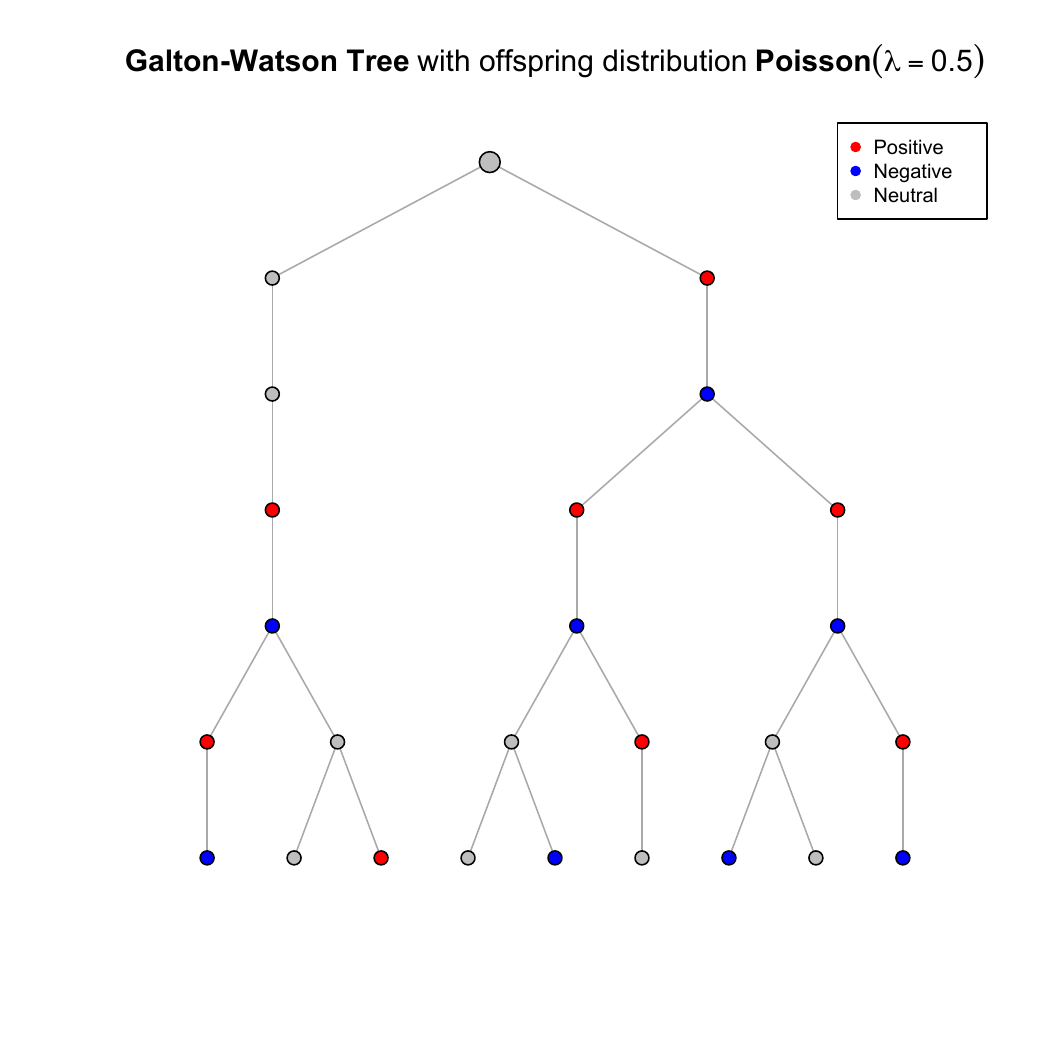}
\includegraphics[width=0.49\linewidth]{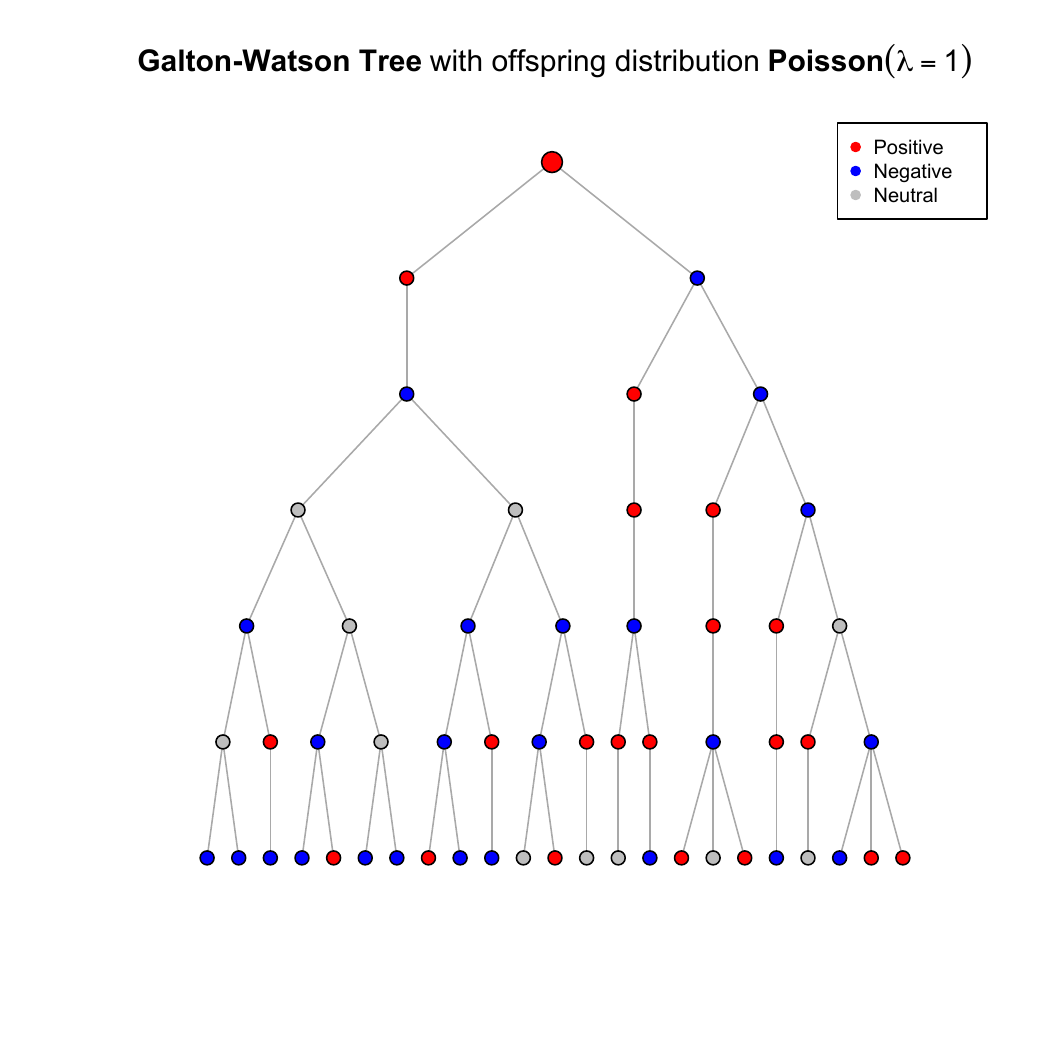}
\includegraphics[width=0.49\linewidth]{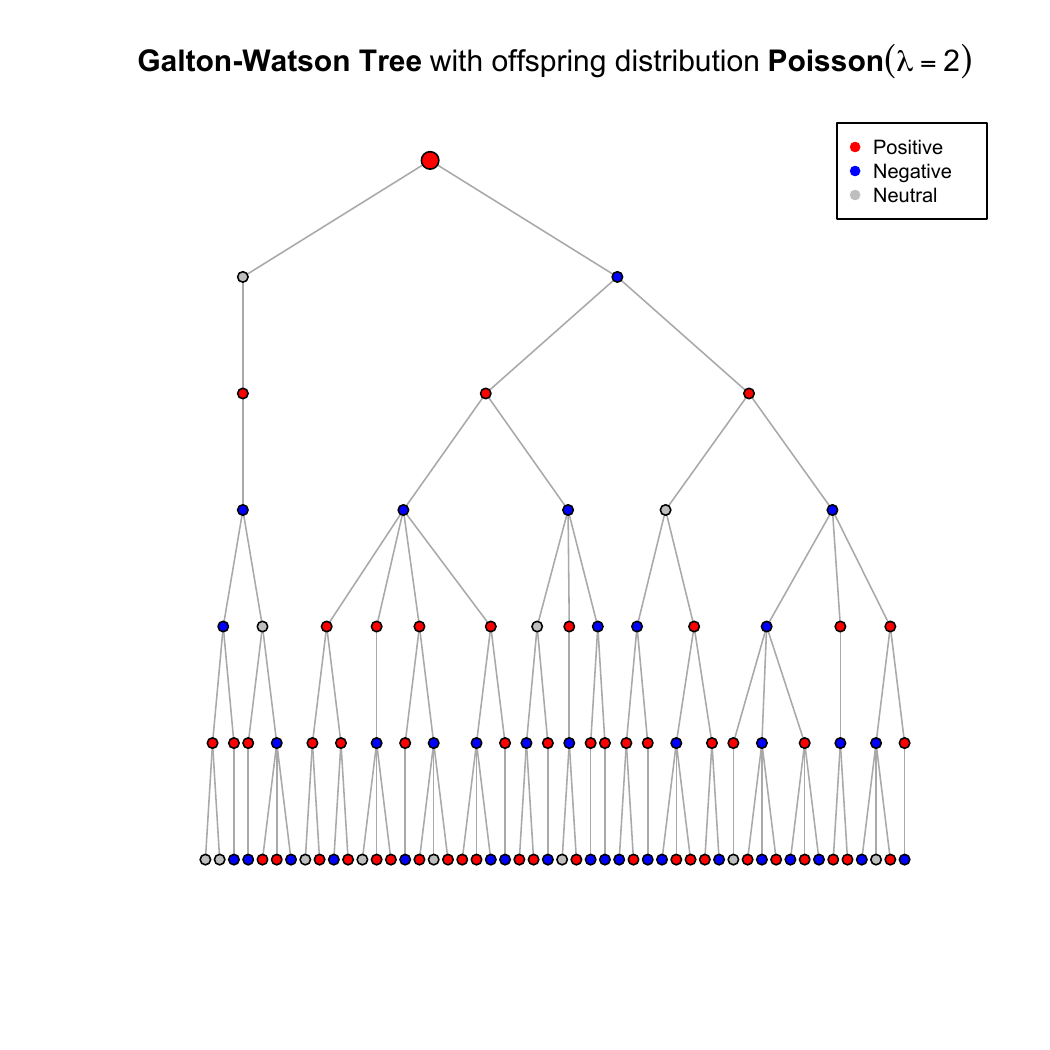}
\end{center}
\vspace{-1.5cm}
\caption{Four realisations of generations $0,\ldots,6$ of an infinite rooted Galton-Watson tree with an offspring distribution that is Poisson$(\lambda)$ with $\lambda = 0.1, 0.5, 1, 2$, respectively. Indicated are the locations of the positive (= red), neutral (= circle) and negative (= blue) vertices \emph{in the infinite tree}. Negative vertices have a tendency to be adjacent to positive vertices. Smaller values of $\lambda$ tend to produce a higher proportion of neutral vertices, while larger values of $\lambda$ tend to produce a higher proportion of positive vertices. In addition, for small values of $\lambda$ the root is almost always positive or neutral, while for large values of $\lambda$ the probability of the root being negative increases, although it remains smaller than the probability of the root being positive or neutral. This is consistent with the findings on the friendship paradox in sparse Erd\H{o}s-R\'enyi random graphs reported in \cite{HHP1}. Indeed, the local limit of the Erd\H{o}s-R\'enyi random graph with edge density $\lambda/n$ is the rooted Galton-Watson tree with offspring distribution Poisson$(\lambda)$ (see \cite[Theorem 2.18]{vdH2}).}
\label{fig:sim}
\end{figure}


\subsubsection{Density of vertex types in infinite Galton-Watson trees}

For $m \in \N_0$, let $\mathcal{T}_\infty^m=(V(\mathcal{T}_\infty^m),E(\mathcal{T}_\infty^m))$ be the induced subtree of $\mathcal{T}_\infty$ with vertex set
\[
V(\mathcal{T}_\infty^m) = \big\{u \in V(\mathcal{T}_\infty)\colon\, \mathrm{dist}(\phi,u)\leq m\big\},
\]
where $\mathrm{dist}(\cdot,\cdot)$ denotes the graph distance. Let $S = \{-,0,+\}$ denote the set of types of the vertices. Put $N_m = |V(\mathcal{T}_\infty^m)|$ and
\[
N_m^\chi = \big|\big\{u\in V(\mathcal{T}_\infty^m)\colon\,\mbox{$u$ has type $\chi$ in $\mathcal{T}_\infty$}\big\}\big|,
\qquad \chi \in S.
\]
The latter counts the number of vertices of type $\chi \in S$ in the infinite tree $\mathcal{T}_\infty$ that are at distance $\leq m$ from the root. To investigate the significance of the friendship paradox for infinite trees, we present the analogue of Definition \ref{def:finite-significance} for infinite trees. 

\begin{definition}{\bf [Significance for infinite trees]}
\label{def:infinite-significance}
{\rm  For an infinite tree $\mathcal{T}_\infty$, the friendship paradox is said to be \textit{significant} when $\liminf_{m\to\infty} (N_m^+ - N_m^-)/N_m \geq 0$, and \textit{insignificant} otherwise. It is said to be \textit{strictly significant} when $\liminf_{m\to\infty} (N_m^+ - N_m^-)/N_m > 0$.} \hfill$\spadesuit$
\end{definition}

Our second theorem identifies the densities of positive, neutral and negative vertices in an infinite Galton-Watson tree with offspring distribution $p=(p_k)_{k\in\N_0}$ satisfying 
\begin{equation}
\label{eq:pcond}
 \sum_{k\in\N} (k \log k)\, p_k <\infty, \qquad p_0=0, \quad p_1<1.
\end{equation}
The first condition is used in order to be able to apply the Kesten-Stigum theorem \cite{KS} and obtain a non-degenerate limit for the normalised population size of each generation. This plays a key role in understanding the asymptotic behaviour of the different vertex types (see, for instance, \eqref{eq:KSthm*} in the proof of Theorem~\ref{thm:densities-on-infinite-trees}). The assumptions $p_0=0$ and $p_1<1$ guarantee almost sure survival of the tree, and exclude the trivial case where every vertex has exactly one offspring. It would be interesting to investigate whether analogues of our results remain valid under weaker assumptions, for instance, under the supercritical condition $\mu = \sum_{k\in\N_0} k p_k >1$, conditionally on survival.

Write $\tilde{p} = (\tilde{p}_k)_{k\in\N_0}$ to denote the \emph{size-biased} offspring distribution given by $\tilde{p}_k = kp_k/\mu$.

\begin{theorem}{\bf [Densities of vertex types for infinite trees]}
\label{thm:densities-on-infinite-trees}
\begin{itemize}
\item[(a)] 
Let $\mathcal{T}_\infty$ be the infinite Galton-Watson tree with offspring distribution $p$ satisfying \eqref{eq:pcond}. Then, for each $\chi \in S$, there exists $f^\chi$  such that 
\[
\lim_{m\to\infty} \frac{N_m^\chi}{N_m} = f^\chi \quad \PP\text{-}\mathrm{a.s.}
\] 
\item[(b)] 
For $k \in \N$, let $S_k= \sum_{l=1}^k X_l$, where $(X_l)_{l=1}^k$ are drawn independently from $p$, and let $\tilde{X}$ be drawn independently from $\tilde{p}$. Then
\[
f^\chi = \sum_{k \in \N_0} p_k\,\PP\Big({\rm sign}\big[\tX+S_k-k(k+1)\big] = \chi\Big).
\]
\end{itemize}
\end{theorem}

The intuition behind Theorem \ref{thm:densities-on-infinite-trees} is as follows. The neighbours of $u$ are its parent and its children. The offspring of the parent is drawn from the size-biased distribution, and is distributed as $\tX$. The offspring of the children are independent copies of $X$. Conditioned on $X_u=k$, with $X_u$ the number of offspring of vertex $u$,  the degree of $u$ is $k+1$, and the total degree of the children of $u$ is distributed as $S_k+k$, where $S_k$ is independent of everything else. The type of $u$ is thus determined in distribution by the sign of 
\[
\frac{1}{k+1}\big(\tX+1+S_k+k\big)-(k+1)=\frac{1}{k+1}\big(\tX+S_k\big)-k,
\]
i.e., by 
\[
{\rm sign}\big[\tX+S_k-k(k+1)\big].
\]

In contrast to finite trees, not all infinite trees are significant.

\begin{example}{\bf [Not all infinite trees are significant]}
\label{ex:GW-non-significant}
{\rm Let $\mathcal{T}_\infty$ be a Galton-Watson tree with an offspring distribution $p$.
\begin{itemize}
\item[(a)] Let $p_1= q \in (0,1)$, $p_a = 1-q$ for some $a \in \N\setminus\{1\}$, and $p_k = 0$ for all $k \notin \{1,a\}$. Then for $q\geq\tfrac12$ and $a$ large enough the friendship paradox for $\mathcal{T}_\infty$ is $\PP$-$\mathrm{a.s.}$ significant, while for $q<\tfrac12$ and $a$ large enough the friendship paradox for $\mathcal{T}_\infty$ is $\PP$-$\mathrm{a.s.}$ insignificant. In addition, it is $\PP$-$\mathrm{a.s.}$ strictly significant for $q>\tfrac12$ and $a$ large enough.
\item[(b)] Let $p_1 = q \in (0,1)$, $p_{2} =\hat{q} \in (0,1)$, $p_a = 1-q-\hat{q}\in (0,1)$ for some $a \in \N\setminus\{1,2\}$, and $p_k = 0$ for all $k \notin \{1,2,a\}$. Then for $q+\hat{q}\geq\tfrac12$ and $a$ large enough the friendship paradox for $\mathcal{T}_\infty$ is $\PP$-$\mathrm{a.s.}$ significant, while for $q+\hat{q}<\tfrac12$ and $a$ large enough the friendship paradox for $\mathcal{T}_\infty$ is $\PP$-$\mathrm{a.s.}$ insignificant. In addition, it is $\PP$-$\mathrm{a.s.}$ strictly significant for $q+\hat{q}>\tfrac12$ and $a$ large enough. \hfill$\spadesuit$
\end{itemize}
}
\end{example}

In Example~\ref{ex:GW-non-significant}(a), all (non-root) vertices with $a$ offspring are either negative or neutral. They can be neutral only if their parent and all their offspring have $a$ offspring themselves, and their parent is not the root. This probability is arbitrarily small when $a$ is large enough. All (non-root) vertices with $1$ offspring are either positive or neutral. They are positive if either their parent or their child have $a$ offspring, and their parent is not the root. By taking $a$ large enough, we can make the fraction of vertices whose parent has $a$ offspring arbitrarily large. Then, most of the vertices with $1$ offspring are positive and most vertices with $a$ offspring are negative. This makes the friendship paradox asymptotically (as $a\to\infty$) significant for $q \geq \tfrac12$ and insignificant for $q < \tfrac12$. Example~\ref{ex:GW-non-significant}(b) is similar, except that now the vertices of ``intermediate'' degree $2$ can be positive, negative or neutral. It turns out that this does not change the result.


\subsubsection{Correlations of vertex types in infinite Galton-Watson trees}

It is interesting to investigate how the positive, neutral and negative vertices are arranged on the tree. A natural question is whether negative vertices tend to be adjacent to positive vertices or not. Figure~\ref{fig:sim} suggests yes. 

Our third theorem quantifies the correlations between the vertex types. For convenience, suppose that the edges in the tree are directed from parents to children. Define
\[
N_m^{\tilde{\chi}\chi} = \big|\big\{(u,ui)\in E(\mathcal{T}_\infty^m)\colon\, 
\mbox{$u$ has type $\tchi$ and $ui$ has type $\chi$ in $\mathcal{T}_\infty$}\big\}\big|,
\]
where $(u,ui)$ denotes a directed edge from $u$ to $ui$, and $E(\mathcal{T}_\infty^m)$ now denotes the set of all directed edges. In what follows we use tilde-notations like $\tchi,\tk$ for the parent vertex. 

\begin{theorem}{\bf [Correlations of vertex types for infinite trees]}
\label{thm:correlations-on-infinite-trees}
\begin{itemize}
\item[(a)] 
Let $\mathcal{T}_\infty$ be the infinite Galton-Watson tree with offspring distribution $p$ satisfying \eqref{eq:pcond}. Then, for each $\tchi,\chi \in S$, there exists $f^{\tchi\chi}$ such that 
\[
\lim_{m\to\infty} \frac{N_m^{\tchi\chi}}{N_m} = f^{\tchi\chi} \quad \PP\text{-}\mathrm{a.s.}
\] 
\item[(b)] 
For $k \in \N$,
\[
\begin{aligned}
f^{\tchi\chi} &= \sum_{\tk \in \N_0} \sum_{k \in \N_0} \tilde{p}_{\tilde{k}}\, p_k\,
\PP\Big({\rm sign}\big[\tX+{S}_{\tilde{k}-1}+k-\tilde{k}(\tilde{k}+1)\big]=\tchi\Big)\\
&\qquad \times \PP\Big({\rm sign}\big[\tilde{k}+S_k-k(k+1)\big]=\chi\Big),
\end{aligned}
\]
where $S_k$ and $\tilde{X}$ are as defined in Theorem \ref{thm:densities-on-infinite-trees}.
\end{itemize}
\end{theorem}

\begin{remark}{\bf [Original expression for correlation]}
{\rm Let $X$ be a random variable drawn independently from $p$, and let $\tX^*$ be an independent copy of $\tX$. Let $S_k^*$ be the sum of $k$ i.i.d.\ random variables drawn from $p$, independent of $X$, $\tX$, $\tX^*$ and ${S}_{\tilde{k}-1}$. The proof of Theorem \ref{thm:correlations-on-infinite-trees} shows that
\begin{align}
\label{AP-fxx}
&f^{\tchi\chi}\nonumber\\ 
&=\sum_{\tk\in\N_0} \sum_{k\in\N_0}  \tilde{p}_{\tk}\,p_k\,
\PP\Big({\rm sign}\big[\tX^*+k+{S}_{\tilde{k}-1}-\tk(\tk+1)\big]=\tchi
\, ,\,{\rm sign}\big[\tk+S_k^* -k(k+1)\big]=\chi\Big)\nonumber\\
&=\PP\Big({\rm sign}\big[\tX^*+X+{S}_{\tX-1}-\tX(\tX+1)\big]=\tchi
\, ,\,{\rm sign}\big[\tX+S_X^* -X(X+1)\big]=\chi\Big),
\end{align}
from which the formula in Theorem~\ref{thm:correlations-on-infinite-trees}(b) follows.
}\hfill$\spadesuit$ 
\end{remark}

The intuition behind Theorem \ref{thm:correlations-on-infinite-trees} is as follows. Consider an edge between a parent $u$ and one of its children $v$. The neighbours of $u$ are its parent and its children (including $v$), and the neighbours of $v$ are its parent $u$ and its own children. The number of offspring of $u$ (as a parent of $v$) and of its own parent are drawn independently from the size-biased distribution, and are distributed as $\tX$ and $\tX^*$, respectively. The number of offspring of $v$ is an independent copy of $X$. The total degree of the children of $u$ is distributed as $X + S_{\tX-1}+\tX$, while the total degree of the children of $v$ is distributed as $S_X^*+X$. Thus, the type of $u$ is determined in distribution by the sign of
\[
\frac{1}{\tX+1}\big(\tX^*+1+X+{S}_{\tX-1}+\tX\big)-(\tX+1)
= \dfrac{1}{\tX+1}\big(\tX^*+X+{S}_{\tX-1}-\tX(\tX+1)\big),
\]
i.e., by 
\[
\tX^*+X+{S}_{\tX-1}-\tX(\tX+1).
\]
Also the type of $v$ is determined in distribution by the sign of 
\[
\frac{1}{X+1}\big(\tX+1+S_X^*+X \big)-(X+1) 
= \frac{1}{X+1}\big(\tX+S_X^* -X(X+1)\big),
\]
i.e., by 
\[
\tX+S_X^* -X(X+1).
\]

Note that the left- and right-marginals of $f^{\tchi\chi}$ are \emph{different}. Indeed, 
\[
\sum_{\tchi\in S}f^{\tchi\chi} = \sum_{\tk\in\N_0} \sum_{k\in\N_0} \tilde{p}_{\tilde{k}}\, p_k\,
\PP\Big({\rm sign}\big[\tilde{k}+S_k-k(k+1)\big]=\chi\Big) = f^\chi
\]
is the density of \emph{childen} of type $\chi$, while 
\begin{align*}
\sum_{\chi\in S} f^{\tchi\chi}
&=\sum_{\tk\in\N_0}\sum_{k\in\N_0} \tilde{p}_{\tilde{k}}\, p_k\,
\PP\Big({\rm sign}\big[\tX+{S}_{\tilde{k}-1}+k-\tilde{k}(\tilde{k}+1)\big]=\tchi\Big)\\
&=\sum_{\tk\in\N_0} \tilde{p}_{\tilde{k}}\,
\PP\Big({\rm sign}\big[\tX+{S}_{\tilde{k}-1}+X-\tilde{k}(\tilde{k}+1)\big]=\tchi\Big)\\
&= \sum_{\tk\in\N_0}  \tilde{p}_{\tilde{k}}\,
\PP\Big({\rm sign}\big[\tX+{S}_{\tilde{k}}-\tilde{k}(\tilde{k}+1)\big]=\tchi\Big)
=: \tf^{\tchi}
\end{align*}
is the density of \emph{parents} of type $\tchi$, with $X$ a random variable drawn from $p$ and independent of everything else. The latter density is size-biased because each parent $u$ counts $X_u$ times. 

Our fourth and last theorem states that positive vertices are negatively correlated subject to a certain condition on the offspring distribution $p$.

\begin{theorem}{\bf [A sufficient criterion for negative correlation]}
\label{thm:plusplus-correlations-on-infinite-trees}
Let $\mathcal{T}_\infty$ be the infinite Galton-Watson tree with offspring distribution $p$ satisfying \eqref{eq:pcond}. Let $S_k$ be as defined in Theorem \ref{thm:densities-on-infinite-trees}. If, for each $\tk$ in the support of $\tilde{p}$, 
\begin{equation}
\label{eq:mono}
k \mapsto \PP\big(\tk + S_k - k(k+1)>0\big) \text{ is non-increasing on the support of $p$},
\end{equation}
then $f^{++} \leq \tf^+ f^+$.
\end{theorem}

Let $\mathcal{T}_\infty$ be the infinite Galton-Watson tree with offspring distribution Poisson$(\lambda)$. Define
\begin{equation*}
f(\tilde{k}, k, \lambda) = \PP\big(\tilde{k} + S_k - k(k+1) > 0\big),
\end{equation*}
where $S_k$ denotes the sum of $k$ i.i.d.\ Poisson$(\lambda)$ random variables. Checking correlation in this setting is analytically challenging, and verifying the monotonicity condition \eqref{eq:mono} is nontrivial as well. However, it is straightforward to check condition \eqref{eq:mono} numerically by using software such as R. In particular, computations show that the condition fails when $6.39 \leq \lambda \leq 41.17$: for example, $f(1,1,\lambda) < f(1,2,\lambda)$ in this range. Therefore, no conclusion about the correlation can be made based on Theorem \ref{thm:plusplus-correlations-on-infinite-trees}. On the other hand, numerical evidence (as shown in Figure~\ref{fig:sim2}) suggests that condition \eqref{eq:mono} does hold for smaller values of $\lambda$. Hence, in these cases, we expect that positive vertices are negatively correlated.

\begin{figure}[htbp]
\begin{center}
\includegraphics[width=0.49\linewidth]{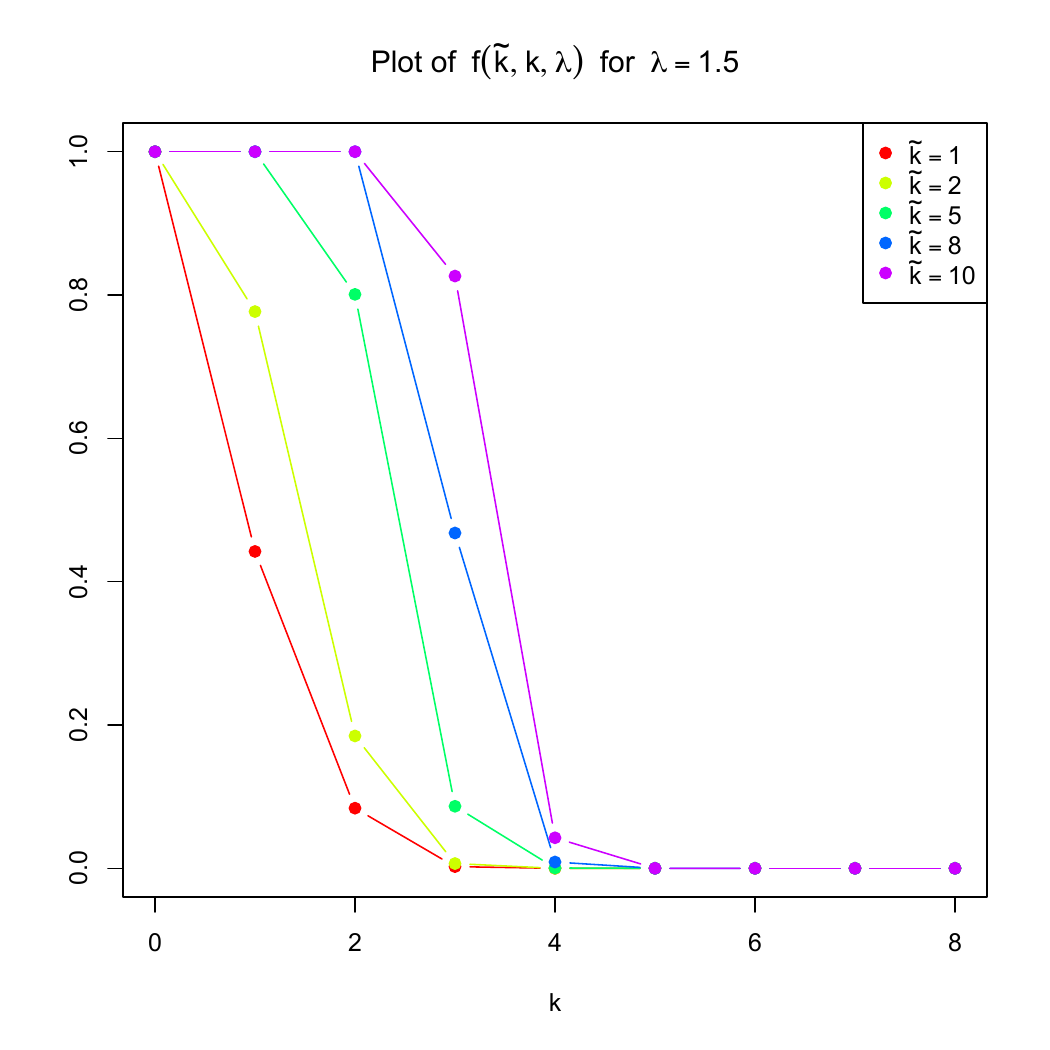}
\includegraphics[width=0.49\linewidth]{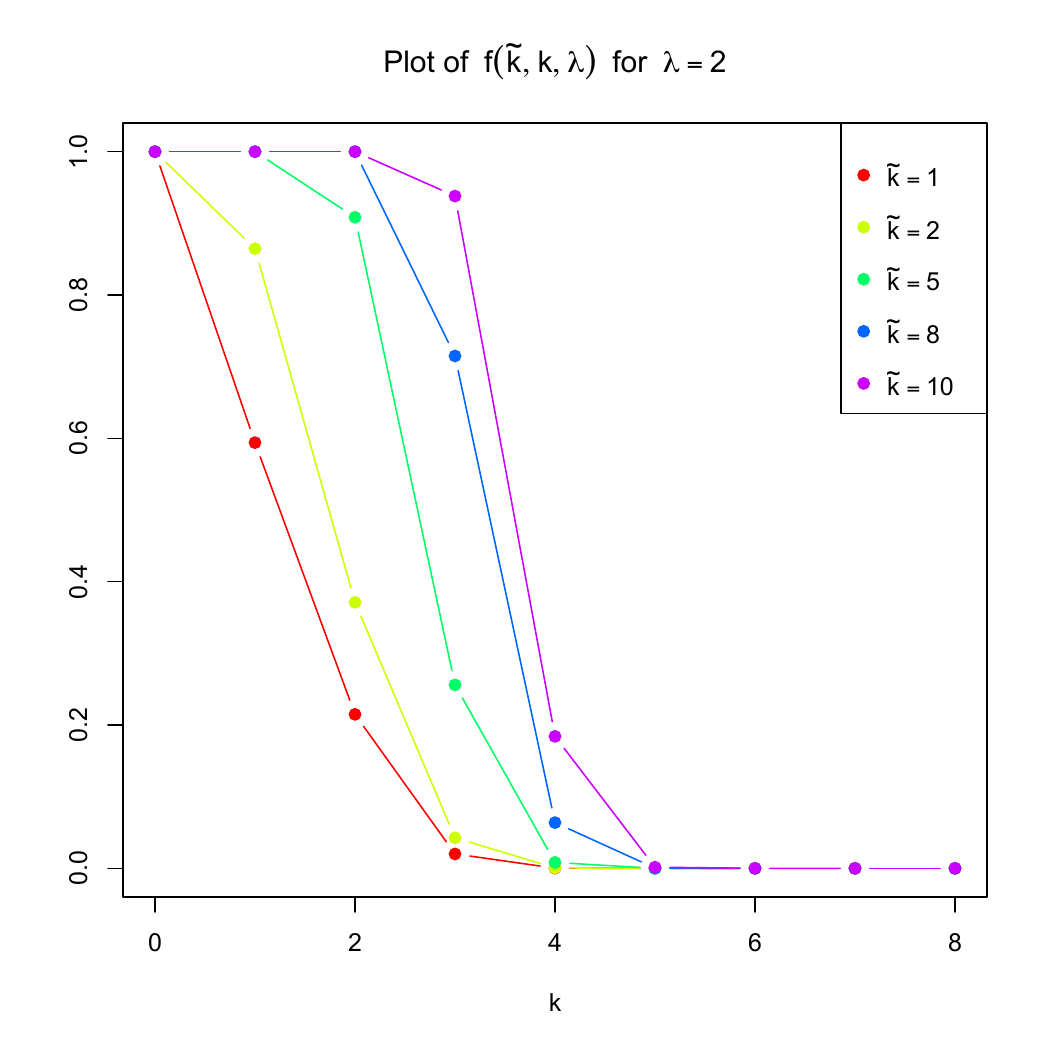}
\includegraphics[width=0.49\linewidth]{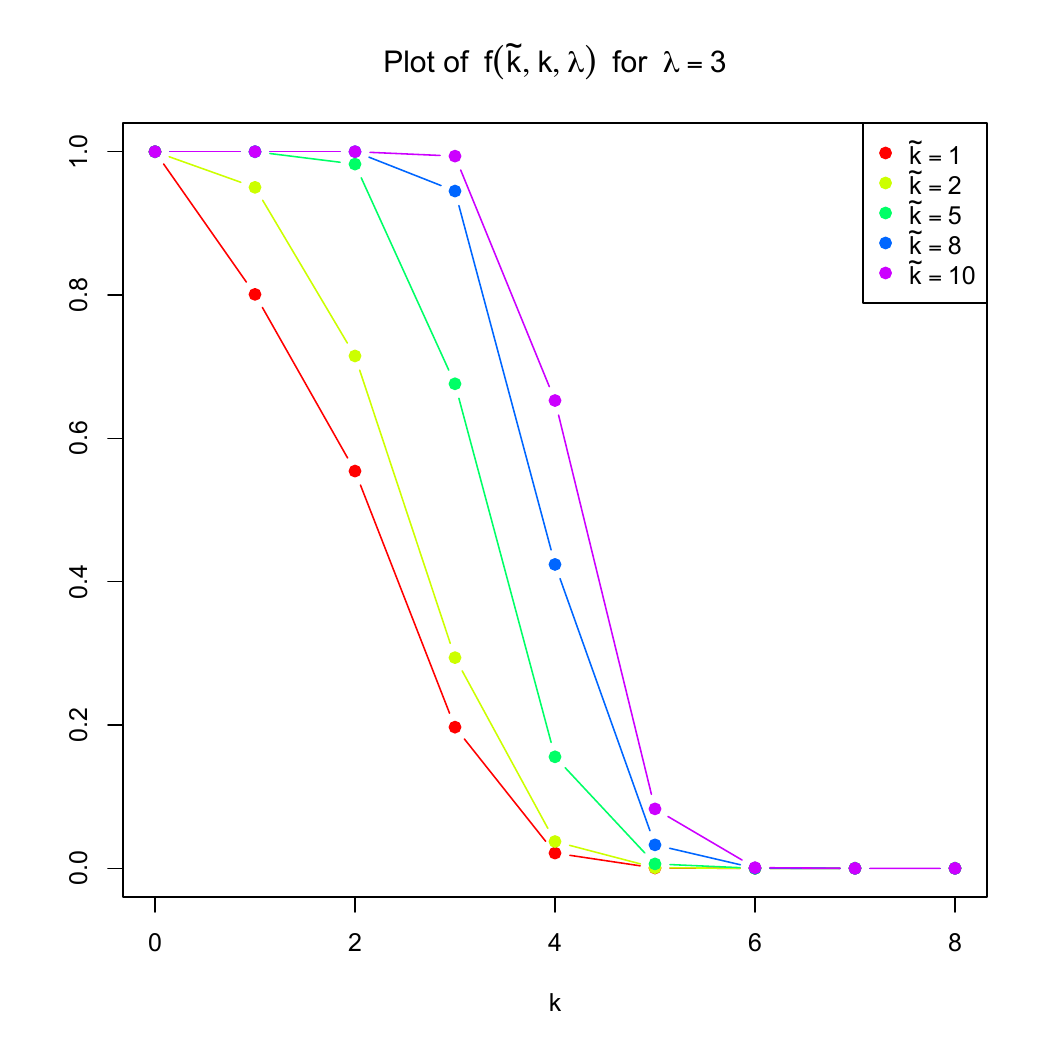}
\includegraphics[width=0.49\linewidth]{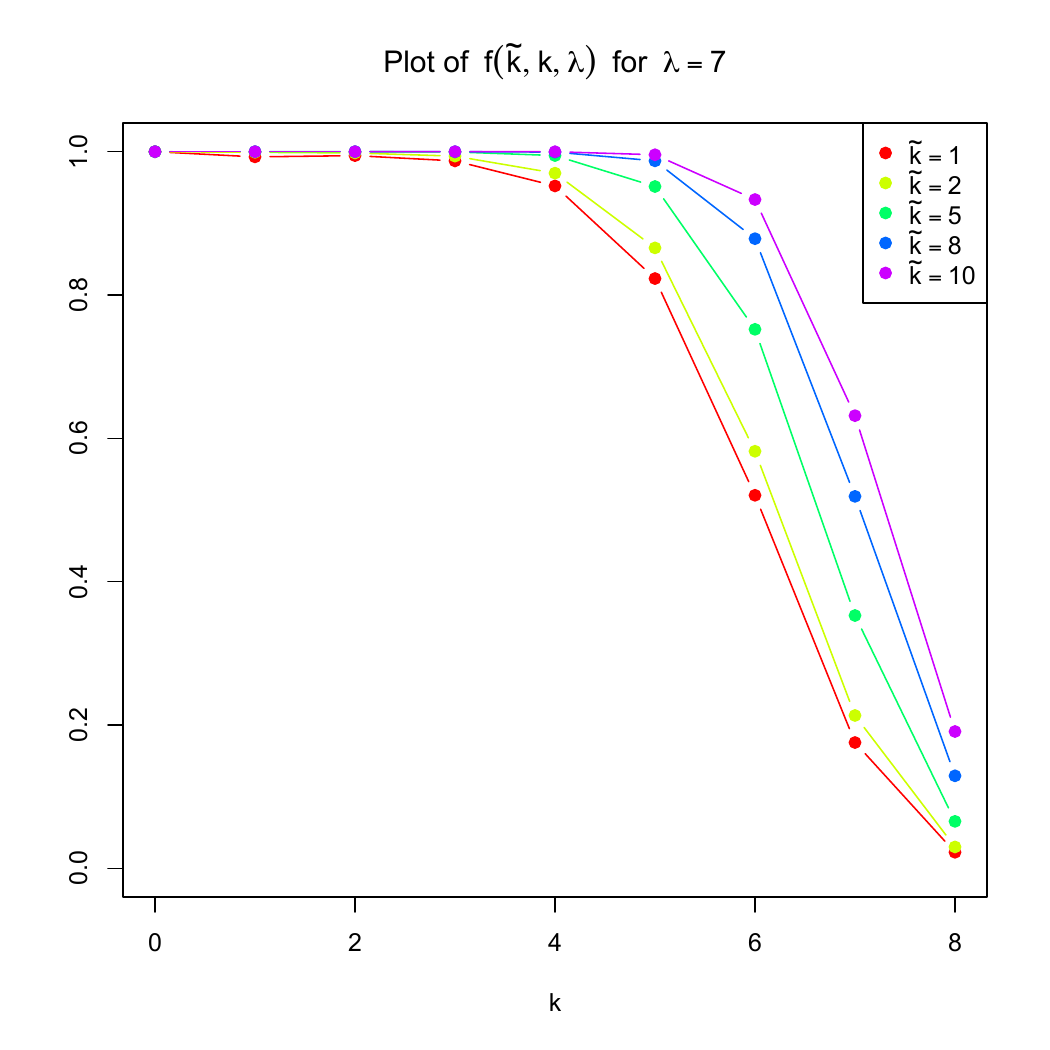}
\end{center}
\vspace{-0.4cm}
\caption{Plots of $k \mapsto f(\tilde{k}, k, \lambda) = \PP(\tilde{k} + S_k - k(k+1) > 0)$ for $\tilde{k} \in \{1,2,5,8,10\}$ and $\lambda \in\{1.5,2,3,7\}$. The function is non-increasing in $k$ for all displayed values, except for $\lambda=7$ and $\tilde{k}=1$, where a slight increase is observed from $k=1$ to $k=2$.}
\label{fig:sim2}
\end{figure}

We close with two examples where the correlation is negative regardless of the condition in \eqref{eq:mono}.

\begin{example}{\bf [Correlation in Galton-Watson trees]}
\label{AP-example}
{\rm
\begin{itemize}
\item[(a)] 
If $p_1\in (0,1)$, $p_a = 1-p_1$ for some $a \in \N\setminus\{1\}$, and $p_k = 0$ for all $k \notin \{1,a\}$, then $f^{++}<\tilde{f}^{+}f^{+}$.
\item[(b)] 
If $p_1\in (0,\frac{1}{2})$, $p_2\in (0,1-p_1)$, $p_a = 1-p_1-p_2\in (0,1)$ for some $a \in \N\setminus\{1,2,3,4\}$, and $p_k = 0$ for all $k \notin \{1,2,a\}$, then $f^{++}<\tilde{f}^{+}f^{+}$.
\hfill$\spadesuit$
\end{itemize}
}
\end{example}

Negative correlation is intuitively plausible, because the presence of a positive vertex typically makes it harder for the neighbouring vertices to also be positive. Negative correlation may have several combinatorial reasons, rather than being a consequence of one specific propagation mechanism.


\subsection{Discussion and open problems}
\label{sec:DiscPro}

\medskip\noindent
{\bf 1.} 
Theorem~\ref{thm:finite-trees-significant} gives a lower bound of the difference of the numbers of positive and negative vertices in an \emph{arbitrary finite tree}. It shows that this difference is at least 2, implying significance, and is bounded below by $2$ plus the total surplus of edges at branching points in the tree. The bound is sharp for instance when the tree consists of 6 vertices, of which 4 are leaves and 2 have degree 3. 

\medskip\noindent
{\bf 2.}
Theorem~\ref{thm:densities-on-infinite-trees} identifies the densities of positive, neutral and negative vertices in an \emph{infinite Galton-Watson tree} under mild conditions on the offspring distribution. Interestingly, the difference of the densities of positive and negative vertices is not always non-negative, and can have both signs depending on the choice of the offspring distribution. Theorem~\ref{thm:correlations-on-infinite-trees} identifies the densities of the edges having two given types of vertices at their ends. Theorem~\ref{thm:plusplus-correlations-on-infinite-trees} gives a sufficient condition under which the positive vertices are negatively correlated along edges, while Example~\ref{AP-example} shows that this condition is not necessary. It remains open to find an example   where the positive vertices are positively correlated along edges. 

\medskip\noindent
{\bf 3.} 
The formulas for the densities of the vertex-types and the edge-types are explicit and numerically easy to implement. They provide a \emph{quantification} of the friendship-bias in infinite Galton-Watson trees. Analytically they are not easy to manipulate, which is why a \emph{full classification} of significance remains an open problem. It would be interesting to see whether similar limiting results could be obtained for other random trees, e.g.\ those arising as local limits of configuration model random graphs and preferential attachment random graphs \cite[Chapters 7--8]{vdH1}.

\medskip\noindent
{\bf 4.} 
In \cite{HHP2} we looked at deeper levels of the friendship-bias, i.e., not only between vertices at distance 1 but at distance $k \in \N$. We found that different scenarios are possible when $n \to \infty$ and $k \to \infty$ jointly. It would be interesting to see what would happen on finite and infinite trees. 

\medskip\noindent
{\bf 5.} 
A tantalising question is the following. In the infinite Galton-Watson tree, for those cases where the positive vertices are negatively correlated, can we prove that they do not percolate? Or conversely, for those cases (if at all) where the positive vertices are positively correlated, can we prove that they do percolate? (Obviously, the answer must also depend on the offspring distribution of the tree.) Underlying these questions is the following: Can we establish a comparison with i.i.d.\ percolation on the same tree via coupling?


\subsection{Outline}
\label{sec:outline}

Theorems \ref{thm:finite-trees-significant}, \ref{thm:densities-on-infinite-trees} and \ref{thm:correlations-on-infinite-trees}, \ref{thm:plusplus-correlations-on-infinite-trees} are proved in Sections \ref{sec:finite-trees}--\ref{sec:infinite-trees}, respectively. The proof of Theorem \ref{thm:finite-trees-significant} is based on an exploration construction of the finite tree in which at every sweep at least as many positive vertices are covered as negative vertices. The proofs of Theorems \ref{thm:densities-on-infinite-trees} and \ref{thm:correlations-on-infinite-trees}, \ref{thm:plusplus-correlations-on-infinite-trees} exploit the randomness of the infinite Galton-Watson tree and the fact that the offsprings of its vertices are i.i.d., which allows for a recursion argument.  In Section~\ref{sec:infinite-trees} we explain Example~\ref{ex:GW-non-significant} and Example~\ref{AP-example} in more detail.


\section{Proofs for finite trees}
\label{sec:finite-trees}

In this section we look at arbitrary finite trees. Specifically, we prove Theorem \ref{thm:finite-trees-significant}. The proof uses an iterative construction of the tree through a sequence of growing subtrees for which the number of positive and negative vertices can be controlled. The construction is carried out in Sections~\ref{ss:bb}--\ref{ss:alg}, and is shown in Section \ref{sec:sgn} to imply significance.


\subsection{Building blocks: leaves and paths}
\label{ss:bb}

Before describing the construction, we state two lemmas that are used as building blocks of the construction. We omit the proofs because these follow from direct verification. 

\begin{lemma}
\label{lem1}
In any tree of size $\geq 3$, any leaf is positive.
\end{lemma}

\begin{lemma}
\label{lem2}
Consider a path of $n$ vertices.
\begin{itemize}
\item[(a)] For $n \in \{1,2\}$, any vertex is neutral.
\item[(b)] For $n = 3$, two vertices (the end vertices) are positive  and one vertex (the middle vertex) is negative.
\item[(c)] For $n \geq 4$, two vertices (the end vertices) are positive, two vertices (the neighbours of the end vertices) are negative, and the rest $n-4$ vertices are neutral.
\end{itemize}
\end{lemma}

\noindent
Lemma~\ref{lem2} settles Theorem~\ref{thm:finite-trees-significant}(b). In the next section we prove Theorem~\ref{thm:finite-trees-significant}(a) via an exploration construction of $\Tt_n$.


\subsection{Exploration construction for finite trees}
\label{ss:alg}

We start with the definition of a branching point.

\begin{definition}{\bf [Branching point]}
\label{def:branching-point}
{\rm Vertex $u\in V(\Tt_n)$ is a branching point when $d_u \geq 3$.} 
\end{definition}

By assumption, $\Tt_n$ contains at least one branching point, which can happen only if $n\ge 4$. Suppose that $\mathcal{T}_n = (V(\mathcal{T}_n),E(\mathcal{T}_n))$ is an arbitrary tree with $n \geq 4$ vertices.  Since we may start exploring from any vertex, we choose to start from a branching point $v$, and we set $\phi =v$ to be the root of the tree. In order to prove Theorem \ref{thm:finite-trees-significant}(a), we construct $(\mathcal{T}_n,\phi)$ step by step as follows. 

Starting from the root $\phi$, we successively add vertices in $m\in\N$ steps to obtain a sequence of subtrees 
\[
\mathcal{T}_{n_1} \subseteq \cdots \subseteq \mathcal{T}_{n_m}
\]
with
\[
\mathcal{T}_{n_1} = \mathcal{T}_1 = (\{\phi\},\emptyset)\quad \text{and} \quad \mathcal{T}_{n_m} = \mathcal{T}_{n},
\]
such that at each step $j$ the number of negative vertices added does not exceed the number of positive vertices added. At each step $j \geq 2$, we construct $\Tt_{n_j}$ by adding $n_j - n_{j-1}$ vertices to $\Tt_{n_{j-1}}$, along with all the edges connecting these new vertices to the vertices already present in $\Tt_{n_{j-1}}$, as well as any edges that might exist between the new vertices themselves. Without loss of generality, we can number the vertices in the order in which they are added, and in the Ulam-Harris ordering mentioned in Section \ref{sec:AP-not}. The construction ends when $\mathcal{T}_{n_m}$ with $n_m=n$ is reached.

\medskip\noindent
\underline{\textit{Step 1:} Construct $\mathcal{T}_{n_1} = (\{\phi\},\emptyset)$.} Here, $N_{n_1}^+ = N_{n_1}^- = 0$ and $N_{n_1}^0 = 1$.

\medskip\noindent
\underline{\textit{Step 2:} Construct $\mathcal{T}_{n_2}$} by adding all $d_\phi$ offspring of $\phi$ together with their edges to $\phi$ (see Figure \ref{fig1}). In $\Tt_{n_2}$, the root $\phi$ is negative, and all other vertices are leaves and therefore are positive: 
\begin{align*}
&N_{n_2}^+=N_{n_1}^+ +d_\phi,\\
&N_{n_2}^-=N_{n_1}^-+1,\\
&N_{n_2}^+-N_{n_2}^-=d_\phi-1=1+(d_\phi-2).
\end{align*}

\begin{figure}[htbp]
\centering
\scalebox{1.3}  
{
\begin{tikzpicture}
\definecolor{colour0}{rgb}{0.8019608,0.8019608,0.8019608}
\draw[thick] (1.67,-0.63) -- (0.87,-1.03);
\draw[thick] (1.67,-0.63) -- (1.67,-1.03);
\draw[thick] (1.67,-0.63) -- (1.27,-1.03);
\draw[thick] (1.67,-0.63) -- (2.07,-1.03);
\draw[thick] (1.67,-0.63) -- (2.47,-1.03);
\draw[thick,colour0] (1.27,-1.03) -- (1.67,-1.43);
\draw[thick,colour0] (1.27,-1.03) -- (0.87,-1.43);
\draw[thick,colour0] (1.27,-1.03) -- (1.27,-1.43);
\draw[thick,colour0] (0.87,-1.43) -- (1.27,-1.83);
\draw[thick,colour0] (0.87,-1.43) -- (0.47,-1.83);
\draw[thick,colour0] (1.67,-1.43) -- (1.67,-1.83);
\draw[thick,colour0] (0.47,-1.83) -- (0.47,-2.63);
\draw[thick,colour0] (1.67,-1.83) -- (1.27,-2.23);
\draw[thick,colour0] (1.67,-1.83) -- (2.07,-2.23);
\draw[thick,colour0] (2.07,-2.23) -- (2.47,-2.63);
\draw[thick,colour0] (2.07,-2.23) -- (1.67,-2.63);
\draw[thick,colour0] (0.47,-2.63) -- (0.87,-3.03);
\draw[thick,colour0] (0.47,-2.63) -- (0.07,-3.03);
\draw[thick,colour0] (1.67,-2.63) -- (2.07,-3.03);
\draw[thick,colour0] (1.67,-2.63) -- (1.67,-3.03);
\draw[thick,colour0] (1.67,-2.63) -- (1.27,-3.03);
\filldraw (1.67,-0.63) circle (1.4pt);
\filldraw (1.27,-1.03) circle (1.4pt);
\filldraw (1.67,-1.03) circle (1.4pt);
\filldraw (2.07,-1.03) circle (1.4pt);
\filldraw (2.47,-1.03) circle (1.4pt);
\filldraw (0.87,-1.03) circle (1.4pt);
\filldraw[fill=colour0] (1.27,-1.43) circle (1.4pt);
\filldraw[fill=colour0] (1.27,-1.83) circle (1.4pt);
\filldraw[fill=colour0] (1.27,-2.23) circle (1.4pt);
\filldraw[fill=colour0] (1.67,-1.43) circle (1.4pt);
\filldraw[fill=colour0] (1.67,-1.83) circle (1.4pt);
\filldraw[fill=colour0] (2.07,-2.23) circle (1.4pt);
\filldraw[fill=colour0] (1.67,-2.63) circle (1.4pt);
\filldraw[fill=colour0] (2.47,-2.63) circle (1.4pt);
\filldraw[fill=colour0] (1.67,-3.03) circle (1.4pt);
\filldraw[fill=colour0] (0.87,-1.43) circle (1.4pt);
\filldraw[fill=colour0] (0.47,-1.83) circle (1.4pt);
\filldraw[fill=colour0] (0.47,-2.23) circle (1.4pt);
\filldraw[fill=colour0] (0.47,-2.63) circle (1.4pt);
\filldraw[fill=colour0] (0.87,-3.03) circle (1.4pt);
\filldraw[fill=colour0] (0.07,-3.03) circle (1.4pt);
\filldraw[fill=colour0] (1.27,-3.03) circle (1.4pt);
\filldraw[fill=colour0] (2.07,-3.03) circle (1.4pt);
\node at (1.68,-0.32) {{\small $\phi$}};
\node at (1.27,-4.23) {\small (a)};
\draw[thick, colour0] (7.27,-0.63) -- (6.87,-1.03);
\draw[thick, colour0] (6.87,-1.03) -- (6.47,-1.43);
\draw[thick, colour0] (6.47,-1.43) -- (6.07,-1.83);
\draw[thick, colour0] (7.27,-0.63) -- (7.67,-1.03);
\draw[thick, colour0] (7.67,-1.03) -- (8.07,-1.43);
\draw[thick, colour0](5.27,-2.63) -- (5.67,-3.03);
\draw[thick, colour0] (5.27,-2.63) -- (4.87,-3.03);
\draw[thick] (6.07,-1.83) -- (5.67,-2.23);
\draw[thick] (5.67,-2.23) -- (6.07,-2.63);
\draw[thick] (5.67,-2.23) -- (5.27,-2.63);
\filldraw[fill=colour0] (7.27,-0.63) circle (1.4pt);
\filldraw[fill=colour0] (6.87,-1.03) circle (1.4pt);
\filldraw[fill=colour0] (6.47,-1.43) circle (1.4pt);
\filldraw[fill=colour0] (7.67,-1.03) circle (1.4pt);
\filldraw[fill=colour0] (8.07,-1.43) circle (1.4pt);
\filldraw[fill=colour0] (5.67,-3.03) circle (1.4pt);
\filldraw[fill=colour0] (4.87,-3.03) circle (1.4pt);
\filldraw (6.07,-1.83) circle (1.4pt);
\filldraw (5.67,-2.23) circle (1.4pt);
\filldraw (6.07,-2.63) circle (1.4pt);
\filldraw (5.27,-2.63) circle (1.4pt);
\node at (5.47,-2.03) {{\small $\phi$}};
\node at (6.87,-4.23) {\small (b)};
\end{tikzpicture}
}
\caption{Indicated in black is the subtree $\mathcal{T}_{n_2}$ of the full tree $\mathcal{T}_n$.}
\label{fig1}
\end{figure}
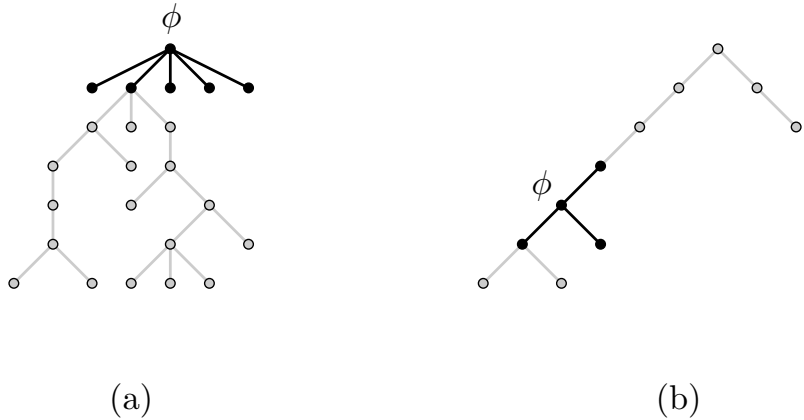

\medskip\noindent 
$\bullet$ Let $j>2$. If $n_{j-1}=n$, then Theorem~\ref{thm:finite-trees-significant}(a) is proved. Otherwise, go to the next step.

\medskip\noindent
\underline{\textit{Step $j$:} Construct $\mathcal{T}_{n_{j}}$.} We have $j>2$, and $\mathcal{T}_{n_{j-1}}$ has been constructed. Take any leaf $u\in V(\mathcal{T}_{n_{j-1}})$ with a degree larger than $1$ in $\mathcal{T}_n$ (like $\phi_1, \phi_2, \phi_3$ and $\phi_5$ in Figure \ref{fig2}).  There are two possible cases:
\begin{itemize}
\item[{\bf Case 1.}] If $u$ is a branching point of $\mathcal{T}_n$ (as $\phi_1$ in Figure \ref{fig2}), then add all its $d_u-1$ offspring to the construction together with their edges to $u$, which become leaves (and thus are positive) in $\Tt_{n_j}$. The vertex $u$ itself may either stay positive or turn into neutral or negative, depending on $d_u$ and on the degree of the parent of $u$. Thus, we can update the number of positive and negative vertices as follows:  
\begin{align*}
N_{n_j}^+ &= N_{n_{j-1}}^+ + d_u - 1 - \ind{\mbox{$u$ is neutral or negative in $\Tt_{n_j}$}},\\
N_{n_j}^- &= N_{n_{j-1}}^-+ \ind{\mbox{$u$ is negative in $\Tt_{n_j}$}},\\
N_{n_j}^+ - N_{n_j}^- &= N_{n_{j-1}}^+ - N_{n_{j-1}}^- + d_u - 1\\ 
&\qquad - \ind{\mbox{$u$ is neutral in $\Tt_{n_j}$}} - \ind{\mbox{$u$ is neutral or negative in $\Tt_{n_j}$}}\\ 
& \ge N_{n_{j-1}}^+ - N_{n_{j-1}}^- + (d_u-3).
\end{align*}
\item[{\bf Case 2.}] If $d_u=2$  (as $\phi_2,\phi_3$ and $\phi_5$ in Figure \ref{fig2}), then continue by adding its descendant(s) and relevant edges to $\Tt_{n_{j-1}}$ until reaching a leaf or a branching point of $\Tt_n$. Suppose that we have added $k$ vertices. Importantly, note that the parent of $u$ is a branching point. Indeed, if it would not, then it would fall under the current Case 2, so we would have continued adding vertices after $u$. Thus, since $d_u=2$, vertex $u$ cannot become negative in $\Tt_{n_j}$. It can become neutral when $k=1$ and the degree of the parent of $u$ is $3$. Furthermore, observe that the $k$-th vertex is positive in $\Tt_{n_j}$ because it is a leaf, and if $k \geq 2$, then the $(k-1)$-th vertex is negative in $\Tt_{n_j}$ because it has degree 2 and its neighbours have degrees 2 and 1. Hence, we get
\begin{align*}
N_{n_j}^+ &= N_{n_{j-1}}^+ + 2 - \ind{\mbox{$k=1$, parent of $u$ has degree 3 in $\Tt_{n_j}$}},\\
N_{n_j}^- &= N_{n_{j-1}}^- + \ind{\mbox{$k\ge 2$}},\\
N_{n_j}^+ - N_{n_j}^- &\geq  N_{n_{j-1}}^+ - N_{n_{j-1}}^-.
\end{align*}
\end{itemize}

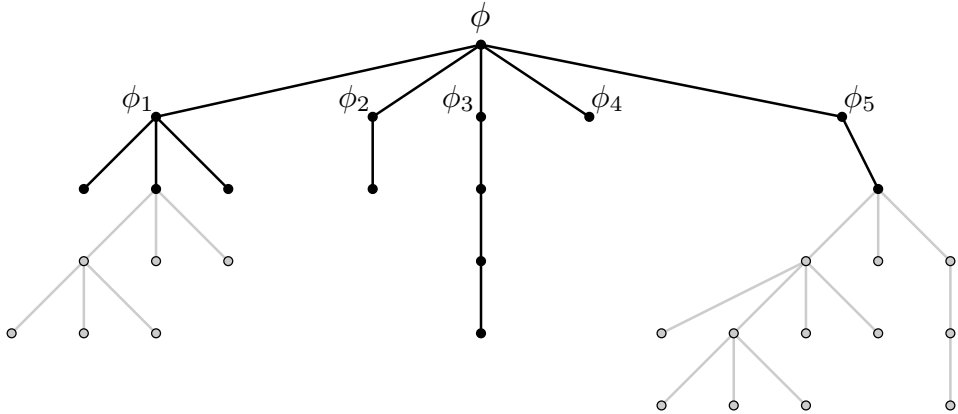
\begin{figure}[htbp]
\centering
\scalebox{1.2}  
{
\begin{tikzpicture}
\definecolor{colour0}{rgb}{0.8019608,0.8019608,0.8019608}
\draw[thick] (6.05,-1.65) -- (7.25,-2.45);
\draw[thick] (6.05,-1.65) -- (6.05,-2.45);
\draw[thick] (6.05,-1.65) -- (4.85,-2.45);
\draw[thick] (6.05,-1.65) -- (2.45,-2.45);
\draw[thick] (6.05,-1.65) -- (10.05,-2.45);
\draw[thick] (6.05,-2.45) -- (6.05,-4.85);
\draw[thick] (4.85,-2.45) -- (4.85,-3.25);
\draw[thick] (2.45,-2.45) -- (1.65,-3.25);
\draw[thick] (2.45,-2.45) -- (2.45,-3.25);
\draw[thick] (2.45,-2.45) -- (3.25,-3.25);
\draw[thick] (10.05,-2.45) -- (10.45,-3.25);
\draw[thick,colour0] (10.45,-3.25) -- (9.65,-4.05);
\draw[thick,colour0] (10.45,-3.25) -- (10.45,-4.05);
\draw[thick,colour0] (10.45,-3.25) -- (11.25,-4.05);
\draw [thick,colour0] (9.65,-4.05) -- (8.85,-4.85);
\draw [thick,colour0] (9.65,-4.05) -- (9.65,-4.85);
\draw [thick,colour0] (9.65,-4.05) -- (10.45,-4.85);
\draw [thick,colour0] (9.65,-4.05) -- (8.05,-4.85);
\draw [thick,colour0] (8.85,-4.85) -- (9.65,-5.65);
\draw [thick,colour0] (8.85,-4.85) -- (8.85,-5.65);
\draw [thick,colour0] (8.85,-4.85) -- (8.05,-5.65);
\draw [thick,colour0] (11.25,-4.05) -- (11.25,-5.65);
\draw [thick,colour0] (2.45,-3.25) -- (1.65,-4.05);
\draw [thick,colour0] (2.45,-3.25) -- (2.45,-4.05);
\draw [thick,colour0] (2.45,-3.25) -- (3.25,-4.05);
\draw [thick,colour0] (1.65,-4.05) -- (2.45,-4.85);
\draw [thick,colour0] (1.65,-4.05) -- (1.65,-4.85);
\draw [thick,colour0] (1.65,-4.05) -- (0.85,-4.85);
\filldraw (6.05,-1.65) circle (1.4pt);
\filldraw (4.85,-2.45) circle (1.4pt);
\filldraw (6.05,-2.45) circle (1.4pt);
\filldraw (7.25,-2.45) circle (1.4pt);
\filldraw (6.05,-3.25) circle (1.4pt);
\filldraw (6.05,-4.05) circle (1.4pt);
\filldraw (6.05,-4.85) circle (1.4pt);
\filldraw (4.85,-3.25) circle (1.4pt);
\filldraw (3.25,-3.25) circle (1.4pt);
\filldraw (2.45,-3.25) circle (1.4pt);
\filldraw (1.65,-3.25) circle (1.4pt);
\filldraw (2.45,-2.45) circle (1.4pt);
\filldraw (10.05,-2.45) circle (1.4pt);
\filldraw (10.45,-3.25) circle (1.4pt);
\filldraw [fill=colour0] (9.65,-4.05) circle (1.4pt);
\filldraw [fill=colour0] (10.45,-4.05) circle (1.4pt);
\filldraw [fill=colour0] (11.25,-4.05) circle (1.4pt);
\filldraw [fill=colour0] (8.05,-4.85) circle (1.4pt);
\filldraw [fill=colour0] (8.85,-4.85) circle (1.4pt);
\filldraw [fill=colour0] (9.65,-4.85) circle (1.4pt);
\filldraw [fill=colour0] (10.45,-4.85) circle (1.4pt);
\filldraw [fill=colour0] (9.65,-5.65) circle (1.4pt);
\filldraw [fill=colour0] (8.85,-5.65) circle (1.4pt);
\filldraw [fill=colour0] (8.05,-5.65) circle (1.4pt);
\filldraw [fill=colour0] (11.25,-4.85) circle (1.4pt);
\filldraw [fill=colour0] (11.25,-5.65) circle (1.4pt);
\filldraw [fill=colour0] (3.25,-4.05) circle (1.4pt);
\filldraw [fill=colour0] (2.45,-4.05) circle (1.4pt);
\filldraw [fill=colour0] (1.65,-4.05) circle (1.4pt);
\filldraw [fill=colour0] (1.65,-4.85) circle (1.4pt);
\filldraw [fill=colour0] (2.45,-4.85) circle (1.4pt);
\filldraw [fill=colour0] (0.85,-4.85) circle (1.4pt);
\node at (6.04,-1.35) {$\phi$};
\node at (2.25,-2.25) {\small $\phi_1$};
\node at (4.65,-2.25) {\small $\phi_2$};
\node at (5.8,-2.25) {\small $\phi_3$};
\node at (7.45,-2.25) {\small $\phi_4$};
\node at (10.25,-2.25) {\small $\phi_5$};
\end{tikzpicture}
}
\caption{Indicated in black is the subtree $\mathcal{T}_{n_3}$ of the full tree $\mathcal{T}_n$.}
\label{fig2}
\end{figure}


\subsection{All finite trees are significant}
\label{sec:sgn}

The above iterative construction has the property that the number of positive vertices is at least as large as the number of negative vertices, at every step $j\geq 1$. Moreover, by construction, looking at the increments of $N_{n_j}^+ - N_{n_j}^-$ when we add the offspring of branching point $u$ (Step 2 and Case 1 in steps $j>2$), we get
\begin{align*}
N_n^+ - N_n^{-}
&= \sum_{j=2}^{m} \left([N_{n_j}^+ - N_{n_j}^-] - [N_{n_{j-1}}^+ - N_{n_{j-1}}^-]\right)\\
&=1 + (d_\phi - 2) + \sum_{j=3}^{m}\left([N_{n_j}^+ - N_{n_j}^-] - [N_{n_{j-1}}^+ - N_{n_{j-1}}^-]\right)
\geq 2 + \sum_{ {u\in V(\Tt_n):} \atop {d_u \geq 3} }(d_u-3),
\end{align*}
where we use that $d_\phi \geq 3$ because the root $\phi$ was chosen to be a branching point. This completes the proof of Theorem \ref{thm:finite-trees-significant}(a).


\section{Proofs for infinite trees}
\label{sec:infinite-trees}

In this section we turn to infinite trees, in particular, \emph{random} infinite trees. In Section~\ref{ss:infGW} we prove Theorem \ref{thm:densities-on-infinite-trees}, while in Section \ref{ss:Ex} we elaborate on Example \ref{ex:GW-non-significant}. In Section \ref{sec:correlations-proof} we prove Theorems~\ref{thm:correlations-on-infinite-trees} and \ref{thm:plusplus-correlations-on-infinite-trees}, while in Section \ref{ap:Ex} we elaborate on Example \ref{AP-example}.


\subsection{Densities of vertex types in infinite Galton-Watson tree}
\label{ss:infGW}

\begin{proof}[Proof of Theorem \ref{thm:densities-on-infinite-trees}]
For $\chi \in S$, we show that
\begin{equation}
\label{eq:limit-tree}
\frac{N^{\chi}_m}{N_m} \stackrel{\PP-\mathrm{a.s.}}{\longrightarrow} f^\chi, \qquad m\to\infty,
\end{equation}
with $f^\chi$ as in Theorem \ref{thm:densities-on-infinite-trees}(b).

\medskip\noindent
{\bf 1.}
For $m \in \N$, let $Z_m = N_m-N_{m-1} = |\partial \Tt_\infty^m|$ denote the number of vertices in generation $m$. Then $(Z_m)_{m \in \N_0}$ is a Galton-Watson process with offspring distribution $p$ and initial value $Z_0=|\partial \Tt_{\infty}^0|=1$. We have $\EE[Z_m] = \mu^m$. Moreover, because of the conditions posed in \eqref{eq:pcond}, the Kesten-Stigum theorem \cite{KS} says that 
\begin{equation}
\label{eq:KSthm*}
\frac{Z_m}{\mu^m} \stackrel{\PP-\mathrm{a.s.}}{\longrightarrow} W_\infty, \qquad m\to\infty,
\end{equation}
where $W_\infty$ is a random variable satisfying $\EE[W_\infty]=1$ and $\PP(W_\infty \in (0,\infty))=1$.

\medskip\noindent
{\bf 2.}
Fix $\chi \in S$. For any vertex $v\in V(\Tt_\infty)$, let $X_v^\chi$ be the number of offspring of $v$ of type $\chi$. Then  we obtain 
\begin{align*}
\EE[X_v^\chi] &= \EE\left[\sum_{i=1}^{X_v}\ind{{\rm sign}\left[X_v+X_{vi1}+X_{vi2}
+\cdots+X_{viX_{vi}}-X_{vi}(X_{vi}+1)\right]=\chi}\right]\\
&= \sum_{\tk\in\N_0} p_{\tk} \sum_{i=1}^{\tk}\EE\left[\ind{{\rm sign}
\left[\tk+X_{vi1}+X_{vi2}+\cdots+X_{viX_{vi}}-X_{vi}(X_{vi}+1)\right]
=\chi}\right]\quad\tag{\scriptsize condition on $X_v=\tk$}\\
&= \sum_{\tk\in\N_0} p_{\tk}\,\tk\,  \EE\left[\ind{{\rm sign}\left[\tk+X_{v11}+X_{v12}
+\cdots+X_{v1X_{v1}}-X_{v1}(X_{v1}+1)\right]=\chi}\right]\quad\tag{\scriptsize $X_{vi}$'s $\PP$-identical to $X_{v1}$}\\
&=\mu \sum_{\tk\in\N_0} \tilde{p}_{\tk}\sum_{k\in\N_0} p_k\, \EE\left[\ind{{\rm sign}
\left[\tk+X_{v11}+X_{v12}+\cdots+X_{v1k}-k(k+1)\right]=\chi}\right]\quad\tag{\scriptsize condition on $X_{v1}=k$}\\
&= \mu\sum_{k\in\N_0} p_k\, \PP\left({\rm sign}\left[\tilde{X} + S_k-k(k+1)\right]
=\chi\right)\quad\tag{\scriptsize full probability formula for $\tX$}\\
&= \mu f^\chi,
\end{align*}
and
\[
\EE[X_v] = \mu.
\]
Let $\mathcal{F}_{m-1}$ be the sigma-field generated by the tree up to generation $m-1$. Write
\[
\partial \mathcal{T}_{\infty}^{m-1} = \{v_1^{(m)} , \ldots , v_{Z_{m-1}}^{(m)}\},
\]
where the vertices are labelled according to the Ulam-Harris ordering. For $1\leq i\leq Z_{m-1}$, let $Y_{m,i}^{\chi}=X_{v_{i}^{(m)}}^\chi$. Define
\[
Z_m^\chi = N_m^\chi-N_{m-1}^\chi 
\]
and note that
\[
Z_m^\chi = \sum_{i=1}^{Z_{m-1}} Y_{m,i}^{\chi}.
\]

\medskip\noindent
{\bf 3.}
We next show the following strong law of large numbers, which is of independent interest.

\begin{lemma}
\label{lem:randomindexSLLN}
$\frac{Z_m^\chi}{Z_{m-1}} \stackrel{\PP-\mathrm{a.s.}}{\longrightarrow} \mu f^\chi$, $m\to \infty$.
\end{lemma}

\begin{proof}
Conditional on $\mathcal{F}_{m-1}$, the random variables $Y_{m,1}^{\chi},\ldots,Y_{m,Z_{m-1}}^{\chi}$ are i.i.d., their common distribution does not depend on $m$, and 
\[
\EE[Y_{m,1}^{\chi}]=\mu f^{\chi} < \infty .
\]
For $a>0$, define
\[
\widetilde{Y}_{m,i}^{\chi} = Y_{m,i}^{\chi}\, \mathbbm{1}_{\{Y_{m,i}^{\chi}\leq a Z_{m-1}\}}, \qquad 1\leq i\leq Z_{m-1}.
\]
Then
\[
\frac{Z_m^{\chi}}{Z_{m-1}}-\mu f^{\chi} = A_m + B_m + D_m,
\]
where
\[
\begin{aligned}
A_m &= \frac1{Z_{m-1}} \sum_{i=1}^{Z_{m-1}} \Big( \widetilde{Y}_{m,i}^{\chi} - \EE[\widetilde{Y}_{m,i}^{\chi}\mid\mathcal{F}_{m-1}] \Big),\\
B_m &= \frac{1}{Z_{m-1}} \sum_{i=1}^{Z_{m-1}} Y_{m,i}^{\chi}\, \mathbbm{1}_{\{Y_{m,i}^{\chi}>aZ_{m-1}\}},\\
D_m &= \EE[\widetilde{Y}_{m,1}^{\chi}\mid\mathcal{F}_{m-1}] -\mu f^{\chi}.
\end{aligned}
\]
Hence, for any $\varepsilon > 0$,
\[
\begin{aligned}
&\PP\left( \Big| \frac{Z_m^\chi}{Z_{m-1}}-\mu f^\chi \Big| >\varepsilon \,\middle|\, \mathcal{F}_{m-1} \right)\\
&\qquad \leq \PP\left( |A_m|>\tfrac{\varepsilon}{3} \,\middle|\, \mathcal{F}_{m-1} \right) 
+ \PP\left( B_m >\tfrac{\varepsilon}{3} \,\middle|\, \mathcal{F}_{m-1} \right) 
+ \mathbbm{1}_{\{|D_m|>\tfrac{\varepsilon}{3}\}} \qquad \mathrm{a.s.}
\end{aligned}
\]
We treat the three terms separately.

\medskip\noindent
(1) By Chebyshev's inequality and conditional independence,
\[
\begin{aligned}
&\PP\left( |A_m|>\tfrac{\varepsilon}{3} \,\middle|\, \mathcal{F}_{m-1} \right)
\leq \frac{9}{\varepsilon^2}\, \EE\left[ A_m^2 \,\middle|\, \mathcal{F}_{m-1} \right]
= \frac{9}{\varepsilon^2 Z_{m-1}}\, \mathbb{V}{\rm ar}\left( \widetilde{Y}_{m,1}^{\chi} \mid \mathcal F_{m-1} \right)\\
&\leq \frac{9}{\varepsilon^2 Z_{m-1}}\, \EE\left[ (\widetilde{Y}_{m,1}^{\chi})^2 \,\middle|\, \mathcal{F}_{m-1} \right]
= \frac{9}{\varepsilon^2 Z_{m-1}}\, \EE\left[(Y_{m,1}^{\chi})^2\, \mathbbm{1}_{\{Y_{m,1}^{\chi}\leq aZ_{m-1}\}} 
\,\middle|\, \mathcal{F}_{m-1} \right] \qquad \mathrm{a.s.}
\end{aligned}
\]
For $k,m\in\N$, define
\[
\Omega_{m}(k) = \left\{\omega\in \Omega\colon\, k^{-1}\mu^{m-1}\leq Z_{m-1}(\omega)\leq k\mu^{m-1} \right\}.
\]
Then, since $\Omega_m(k)\in\mathcal F_{m-1}$,
\[
\begin{aligned}
\PP\left( \left\{|A_m|>\tfrac{\varepsilon}{3}\right\}\cap \Omega_{m}(k) \right)
&
=\EE\Big[ \mathbbm{1}_{\Omega_{m}}(k)\, \PP\left(|A_m|>\tfrac{\varepsilon}{3} \,\middle|\, \mathcal{F}_{m-1} \right) \Big]
\\&
\leq \frac{9}{\varepsilon^2}\, \EE\left[ \mathbbm{1}_{\Omega_m(k)}\frac{1}{Z_{m-1}}\, \EE\left[ (Y_{m,1}^{\chi})^2\, \mathbbm{1}_{\{Y_{m,1}^{\chi}\leq aZ_{m-1}\}}
\,\middle|\, \mathcal{F}_{m-1} \right] \right]
\\&
\leq \frac{9k}{\varepsilon^2}\,\mu^{-(m-1)}\, \EE\left[ (Y_{1,1}^{\chi})^2\, \mathbbm{1}_{\{Y_{1,1}^{\chi}\leq ak\mu^{m-1}\}}
\right].
\end{aligned}
\]
Hence
\[
\begin{aligned}
\sum_{m\in\N} \PP\left( \{|A_m|>\tfrac{\varepsilon}{3}\}\cap\Omega_m(k) \right)
&
\leq \frac{9k}{\varepsilon^2} \sum_{m\in\N} \mu^{-(m-1)}\, \EE\Big[ (Y_{1,1}^{\chi})^2\, \mathbbm{1}_{\{Y_{1,1}^{\chi}\le ak\mu^{m-1}\}} \Big]
\\&
= \frac{9k}{\varepsilon^2}\, \EE\Big[ (Y_{1,1}^{\chi})^2 \sum_{m\in\N} \mu^{-(m-1)}\,\mathbbm{1}_{\{Y_{1,1}^{\chi}\le ak\mu^{m-1}\}} \Big],
\end{aligned}
\]
where the last equality follows from Tonelli's theorem. Now, for every $y\in\N$, if
\[
y\leq ak\mu^{m-1},
\]
then
\[
m-1\geq \log_{\mu}\Big(\frac{y}{ak}\Big).
\]
Hence
\[
\begin{aligned}
\sum_{m\in\N} \PP\left( \{|A_m|>\tfrac{\varepsilon}{3}\}\cap\Omega_m(k) \right)
&
\leq \frac{9k}{\varepsilon^2}\,\EE\bigg[ (Y_{1,1}^{\chi})^2\,\mathbbm{1}_{\{Y_{1,1}^{\chi}>0\}}\, \sum_{m\geq \lceil\log_\mu(Y_{1,1}^{\chi}/(ak))\rceil+1} \mu^{-(m-1)}
\bigg]
\\&
\leq \frac{9k}{\varepsilon^2}\, \EE\bigg[ (Y_{1,1}^{\chi})^2\,\mathbbm{1}_{\{Y_{1,1}^{\chi}>0\}}\,  \mu^{-\log_{\mu}(Y_{1,1}^{\chi}/(ak))}\,(1-\mu^{-1})^{-1} \bigg]
\\&
\leq \frac{9ak^2}{\varepsilon^2(1-\mu^{-1})}\, \EE[Y_{1,1}^{\chi}] <\infty.
\end{aligned}
\]

\medskip\noindent
(2) Since the event $\{B_m>\tfrac{\varepsilon}{3}\}$ implies that at least one of the variables $Y_{m,1}^{\chi},\ldots,Y_{m,Z_{m-1}}^{\chi}$ is larger than $aZ_{m-1}$, we have, by the union bound,
\[
\PP\left( B_m>\tfrac{\varepsilon}{3} \,\middle|\, \mathcal{F}_{m-1}\right)
\leq Z_{m-1}\PP\left( Y_{m,1}^{\chi}>aZ_{m-1} \,\middle|\, \mathcal{F}_{m-1}\right)  \qquad \mathrm{a.s.}
\]
Hence, since $\Omega_m(k)\in\mathcal F_{m-1}$,
\[
\begin{aligned}
\PP\left(\{B_m>\tfrac{\varepsilon}{3}\}\cap\Omega_m(k)\right)
&
= \EE\left[ \mathbbm{1}_{\Omega_m(k)}\,\PP\left( B_m>\tfrac{\varepsilon}{3} \,\middle|\, \mathcal{F}_{m-1} \right) \right]
\\&
\leq \EE\left[ \mathbbm{1}_{\Omega_m(k)}\, Z_{m-1}\, \PP\left( Y_{m,1}^{\chi} > aZ_{m-1} \,\middle|\, \mathcal{F}_{m-1} \right) \right]
\\&
\leq k\mu^{m-1}\,\PP\left( Y_{1,1}^{\chi} > ak^{-1}\mu^{m-1} \right).
\end{aligned}
\]
Consequently,
\[
\begin{aligned}
\sum_{m\in\N} \PP\left( \{B_m>\tfrac{\varepsilon}{3}\}\cap\Omega_m(k) \right)
&
\leq k\, \sum_{m\in\N} \mu^{m-1}\, \PP\left( Y_{1,1}^{\chi} > ak^{-1}\mu^{m-1} \right)
\\&
= k\, \EE\bigg[ \sum_{m\in\N} \mu^{m-1}\, \mathbbm{1}_{\{Y_{1,1}^{\chi}>ak^{-1}\mu^{m-1}\}} \bigg],
\end{aligned}
\]
where the last equality follows from Tonelli's theorem. 
Now, for every $y\in\N$, if
\[
y > ak^{-1}\mu^{m-1},
\]
then
\[
m-1< \log_\mu\Big(\frac{yk}{a}\Big).
\]
Hence
\[
\begin{aligned}
\sum_{m\in\N}\PP\left(\{B_m>\tfrac{\varepsilon}{3}\}\cap\Omega_m(k)\right)
&
\leq
k\, \EE\bigg[\mathbbm{1}_{\{Y_{1,1}^{\chi}>0\}}\sum_{m\in\N:\, m-1<\log_{\mu}(kY_{1,1}^{\chi}/a)} \mu^{m-1} \bigg]
\\&
\leq k\, \EE\bigg[\mathbbm{1}_{\{Y_{1,1}^{\chi}>0\}}\, \frac{\mu^{\log_{\mu}(kY_{1,1}^{\chi}/a)+1}}{\mu-1}\bigg]
\\&
\leq \frac{\mu k^2}{a(\mu-1)}\,\EE[Y_{1,1}^{\chi}] < \infty.
\end{aligned}
\]
 
\medskip\noindent
(3) Note that
\[
D_m = -\EE\left[ Y^{\chi}_{m,1}\mathbbm{1}_{\{Y^{\chi}_{m,1}>aZ_{m-1}\}} \,\middle|\, \mathcal{F}_{m-1} \right].
\]
Hence
\[
\begin{aligned}
|D_m|\, \mathbbm{1}_{\Omega_m(k)}
&
= \EE\left[ Y_{m,1}^{\chi} \, \mathbbm{1}_{\{Y_{m,1}^{\chi}>aZ_{m-1}\}} \,\middle|\, \mathcal{F}_{m-1} \right]\, \mathbbm{1}_{\Omega_m(k)}
\\&
\leq \EE\left[ Y_{m,1}^{\chi}\, \mathbbm{1}_{\{Y_{m,1}^{\chi} > ak^{-1}\mu^{m-1}\}} \,\middle|\, \mathcal{F}_{m-1} \right]  
\\&
= \EE\left[ Y_{1,1}^{\chi}\, \mathbbm{1}_{\{Y_{1,1}^{\chi}>ak^{-1}\mu^{m-1}\}} \right]
\qquad \mathrm{a.s.}
\end{aligned}
\]
Since $\EE[Y_{1,1}^{\chi}]<\infty$, the dominated convergence theorem implies that
\[
\EE\left[ Y_{1,1}^{\chi}\, \mathbbm{1}_{\{Y_{1,1}^{\chi} > ak^{-1}\mu^{m-1}\}} \right] \to 0, \qquad m\to\infty.
\]
Hence, there exists an $m_0(\varepsilon,k)\in\N$ such that, for all $m\geq m_0(\varepsilon,k)$,
\[
|D_m| \, \mathbbm{1}_{\Omega_m(k)} < \tfrac{\varepsilon}{3} \qquad \mathrm{a.s.}
\]
Consequently,
\[
\PP\left( \{|D_m|>\tfrac{\varepsilon}{3}\}\cap\Omega_m(k) \right) = 0
\]
for all sufficiently large $m$, and therefore
\[
\sum_{m\in\N} \PP\left( \{|D_m|>\tfrac{\varepsilon}{3}\}\cap\Omega_m(k) \right) < \infty.
\]

\medskip 
Combining the estimates in (1)--(3), we obtain, for every fixed $k\in\N$,
\[
\sum_{m\in\N} \PP\left( \left\{ \left| \frac{Z_m^\chi}{Z_{m-1}} - \mu f^\chi \right|>\varepsilon \right\} \cap \Omega_m(k) \right) <\infty.
\]
For $k,M\in\N$, put
\[
\Omega_{k,M} = \bigcap_{m\geq M} \Omega_m(k).
\]
Since $\Omega_{k,M}\subseteq \Omega_m(k)$ for every $m\geq M$, it follows that
\[
\sum_{m\geq M} \PP\left( \left\{ \left| \frac{Z_m^\chi}{Z_{m-1}}-\mu f^{\chi} \right|>\varepsilon \right\} \cap \Omega_{k,M} \right) < \infty.
\]
Hence, if
\[
E_{\varepsilon}^{\chi} = \left\{\left\{ \left| \frac{Z_m^{\chi}}{Z_{m-1}}-\mu f^{\chi} \right| > \varepsilon \right\}
\,\,\, \mathrm{i.o.}\right\},
\]
then, by the Borel-Cantelli lemma,
\[
\PP\left(E_{\varepsilon}^{\chi} \cap \Omega_{k,M} \right)=0.
\]
By \eqref{eq:KSthm*},
\[
\PP\left( \bigcup_{k\in\N}\bigcup_{M\in\N}\Omega_{k,M} \right) = 1.
\]
Therefore
\[
\begin{aligned}
\PP\left(E_{\varepsilon}^{\chi}\right)
&=
\PP\left(E_{\varepsilon}^{\chi}\cap\left(\bigcup_{k\in\N}\bigcup_{M\in\N}\Omega_{k,M}\right)\right) 
\leq\sum_{k\in\N}\sum_{M\in\N}\PP\left(E_{\varepsilon}^{\chi}\cap \Omega_{k,M}\right)=0,
\end{aligned}
\]
and so
\[
\PP \left(\left\{ \left| \frac{Z_m^{\chi}}{Z_{m-1}}-\mu f^{\chi} \right| > \varepsilon \right\}
\,\,\, \mathrm{i.o.}\right) = 0.
\]
Since $\varepsilon>0$ was arbitrary, the claim in Lemma~\ref{lem:randomindexSLLN} follows.
\end{proof}

\medskip\noindent
{\bf 4.}
By \eqref{eq:KSthm*},
\[
\frac{Z_m}{Z_{m-1}} \stackrel{\PP-\mathrm{a.s.}}{\longrightarrow} \mu, \qquad m \to \infty.
\]
Combining this with Lemma~\ref{lem:randomindexSLLN}, we get
\begin{align}
\label{eq:convergence-generation}
\frac{Z_m^\chi}{Z_m}
&= \frac{Z_m^\chi}{Z_{m-1}} \, 
\frac{Z_{m-1}}{Z_m} \stackrel{\PP-\mathrm{a.s.}}{\longrightarrow} 
\frac{\mu f^\chi}{\mu} = f^{\chi}, \qquad m\to \infty.
\end{align}
The result in \eqref{eq:convergence-generation} says that the fraction of vertices of type $\chi$ in the \emph{single generation} $m$ converges almost surely to $f^\chi$ as $m\to\infty$. It remains to show that the same holds for the fraction of vertices of type $\chi$ in the \emph{entire tree}. 

\medskip\noindent
{\bf 5.}
First, we show that \eqref{eq:KSthm*} implies
\begin{equation}
\label{eq:KSthm}
\frac{N_m}{\mu^m} \stackrel{\PP-\mathrm{a.s.}}{\longrightarrow} \,\frac{\mu}{\mu-1}\, W_\infty.
\end{equation}
Indeed, pick $M_m$ such that $\log_\mu m \ll M_m \ll m$, and split
\[
N_m = \sum_{j=0}^{m-M_m} Z_j + \sum_{j=m-M_m+1}^m Z_j.
\]
Estimate
\[
\sum_{j=0}^{m-M_m} Z_j \leq (m-M_m+1) Z_{m-M_m},
\]
to get
\begin{align}
\label{new-1}
\frac{1}{\mu^m} \sum_{j=0}^{m-M_m} Z_j \leq \frac{m-M_m+1}{\mu^{M_m}}\,\frac{Z_{m-M_m}}{\mu^{m-M_m}}
\,\stackrel{\PP-\mathrm{a.s.}}{\longrightarrow}\, 0 \times W_\infty = 0.   
\end{align}
Moreover, by the change of variables $k=m-j$, we have
\[
\frac{1}{\mu^m}\sum_{j=m-M_m+1}^m Z_j
=
\sum_{k=0}^{M_m-1}\mu^{-k}\frac{Z_{m-k}}{\mu^{m-k}}.
\]
By \eqref{eq:KSthm*}, there exists an event $\Omega_0\subseteq \Omega$ with
$\PP(\Omega_0)=1$ such that, for every $\omega\in\Omega_0$,
\[
\frac{Z_n(\omega)}{\mu^n}\to W_\infty(\omega), \qquad n\to\infty.
\]
Fix $\omega\in\Omega_0$. Then the sequence $(Z_n(\omega)/\mu^n)_{n\geq 0}$ is bounded, and so there exists a $C(\omega)<\infty$ such that
\[
\frac{Z_n(\omega)}{\mu^n}\leq C(\omega), \qquad n\in \N_0.
\]
Since $M_m\ll m$, we have $M_m-1\leq m$ for all sufficiently large $m$. Therefore, for every sufficiently large $m$ and every integer $k$ such that $0\leq k\leq M_m-1$,
\[
\mu^{-k}\frac{Z_{m-k}(\omega)}{\mu^{m-k}} \leq C(\omega)\,\mu^{-k}.
\]
Since $M_m\to\infty$ and $\sum_{k\in\N_0}C(\omega)\,\mu^{-k}<\infty$, and since for every fixed $k\in\N_0$,
\[
\frac{Z_{m-k}(\omega)}{\mu^{m-k}}\to W_\infty(\omega),\qquad m\to\infty,
\]
the dominated convergence theorem yields
\[
\sum_{k=0}^{M_m-1}\mu^{-k}\frac{Z_{m-k}(\omega)}{\mu^{m-k}}
\to
W_{\infty}(\omega)\sum_{k\in\N_0} \left(\frac{1}{\mu}\right)^k = \frac{\mu}{\mu-1}\, W_\infty(\omega),\qquad m\to\infty.
\]
Thus
\begin{align}
\label{new-2}
\frac{1}{\mu^m} \sum_{j=m-M_m+1}^m Z_j 
\stackrel{\PP-\mathrm{a.s.}}{\longrightarrow} 
  \frac{\mu}{\mu-1}\, W_\infty.
\end{align}
Combining \eqref{new-1} and \eqref{new-2}, we get \eqref{eq:KSthm}. 

\medskip\noindent
{\bf 6.}
Next, we write
\[
N_m^\chi = \sum_{j=0}^m Z_j^{\chi}
= \sum_{j=0}^m \frac{Z_j^{\chi}}{Z_j}\,Z_j,
\]
and split
\[
N_m^\chi 
= R_{m-M_m}\,N_{m-M_m} + \sum_{j=m-M_m+1}^m\frac{Z_j^{\chi}}{Z_j}\,Z_j,
\]
where $\PP(R_{m-M_m} \in [0,1]) = 1$ for all $m$. It follows from \eqref{eq:convergence-generation} that
\[
N_m^\chi \stackrel{\PP-\mathrm{a.s.}}{=} R_{m-M_m}\,N_{m-M_m} + f^\chi(1+o(1))\,[N_m - N_{m-M_m}], \qquad m \to \infty.
\]
Moreover, by \eqref{eq:KSthm},
\[
\frac{N_{m-M_m}}{N_m} \stackrel{\PP-\mathrm{a.s.}}{\longrightarrow} 0.
\]
Hence
\[
\frac{N_m^\chi}{N_m} \stackrel{\PP-\mathrm{a.s.}}{\longrightarrow} f^\chi,
\]
which concludes the proof of Theorem \ref{thm:densities-on-infinite-trees}.
\end{proof}


\subsection{Not all infinite Galton-Watson trees are significant}
\label{ss:Ex}

In this section we show that the infinite Galton-Watson tree may be $\PP\text{-}\mathrm{a.s.}$ significant or $\PP\text{-}\mathrm{a.s.}$ insignificant, depending on its offspring distribution.

\begin{proof}[Proof of the statements in Example~\ref{ex:GW-non-significant}]
Consider a Galton-Watson tree with an offspring distribution $p$. Let $X_v$ denote the size of the offspring of vertex $v$.

\medskip\noindent 
(a) Let $p_1 = q \in (0,1)$, $p_a = 1-q$ for some $a \in \N\setminus\{1\}$, and $p_k = 0$ for all $k \notin \{1,a\}$. Set $\mu = q + a(1-q)$ for the mean offspring. Note that a vertex with minimal degree $2$ cannot be negative, a vertex with maximal degree $a+1$ cannot be positive, while the probability distribution of the number of offspring of the parent vertex of $v$ is size-biased. Then
\[
\begin{aligned}
\PP(v \text{ is positive} \mid X_v = 1) &= 1-q\,\frac{q}{\mu},\\
\PP(v \text{ is neutral} \mid X_v = 1) &= q\,\frac{q}{\mu},\\
\PP(v \text{ is negative} \mid X_v = 1) &= 0,\\
\end{aligned}
\]   
and 
\[
\begin{aligned}
\PP(v \text{ is positive} \mid X_v = a) &= 0,\\
\PP(v \text{ is neutral} \mid X_v = a) &= (1-q)^a\,\frac{a(1-q)}{\mu},\\
\PP(v \text{ is negative} \mid X_v = a) &= 1-(1-q)^a\,\frac{a(1-q)}{\mu}.\\
\end{aligned}
\]   
It follows (via the full probability formula) that
\[
\begin{aligned}
\PP(v \text{ is positive}) &= q - \frac{1}{\mu}\,q^3,\\
\PP(v \text{ is neutral}) &= \frac{1}{\mu}[q^3 + a (1-q)^{a+2}],\\
\PP(v \text{ is negative}) &= (1-q) - \frac{1}{\mu}\,a\,(1-q)^{a+2}.\\
\end{aligned}
\]   
In the limit as $a\to\infty$,
\[
\begin{aligned}
f^+ =\PP(v \text{ is positive}) &\to q,\\
f^0 = \PP(v \text{ is neutral}) &\to 0,\\
f^- = \PP(v \text{ is negative}) &\to 1-q.\\
\end{aligned}
\]   
Hence, from Definition \ref{def:infinite-significance} and Theorem \ref{thm:densities-on-infinite-trees}, for $q\geq\tfrac12$ and $a$ large enough the friendship paradox for the tree is $\PP\text{-}\mathrm{a.s.}$ significant, while for $q<\tfrac12$ and $a$ large enough the friendship paradox for the tree is $\PP\text{-}\mathrm{a.s.}$ insignificant. It is $\PP\text{-}\mathrm{a.s.}$ strictly significant for $q>\tfrac12$ and $a$ large enough.

\medskip\noindent 
(b) Let $p_1 = q \in (0,1)$, $p_{2} =\hat{q} \in (0,1)$, $p_a = 1-q-\hat{q}\in (0,1)$ for some $a \in \N\setminus\{1,\ldots, 4\}$, and $p_k = 0$ for all $k \notin \{1,2,a\}$. Set $\mu = q +2\hat{q} + a(1-q-\hat{q})$ for the mean offspring. Note that a vertex of degree $3$ can be negative if and only if the number of offspring of its three neighbours are either all $1$, or two are $1$ and one is $2$, or one is $1$ and two are $2$. Then, similar to the proof of (a), we have
\[
\begin{aligned}
\PP(v \text{ is positive} \mid X_v = 1) &= 1-q\,\frac{q}{\mu},\\
\PP(v \text{ is neutral} \mid X_v = 1) &= q\,\frac{q}{\mu},\\
\PP(v \text{ is negative} \mid X_v = 1) &= 0,
\end{aligned}
\]   
and 
\[
\begin{aligned}
\PP(v \text{ is positive} \mid X_v = 2) &= 1-\hat{q}^{2}\,\frac{2\hat{q}}{\mu}-3q^{2}\,\frac{q}{\mu}-6q\hat{q}\,\frac{q}{\mu}-\hat{q}^{2}\,\frac{q}{\mu},\\
\PP(v \text{ is neutral} \mid X_v = 2) &= \hat{q}^{2}\,\frac{2\hat{q}}{\mu},\\ 
\PP(v \text{ is negative} \mid X_v = 2) &= q^{2}\,\frac{q}{\mu}+2q\hat{q}\,\frac{q}{\mu}+\hat{q}^{2}\,\frac{q}{\mu}+q^{2}\,\frac{2q}{\mu}+2q\hat{q}\,\frac{2q}{\mu}\\
&=3q^{2}\,\frac{q}{\mu}+6q\hat{q}\,\frac{q}{\mu}+\hat{q}^{2}\,\frac{q}{\mu}.\\
\end{aligned}
\]   
 In addition,
\[
\begin{aligned}
\PP(v \text{ is positive} \mid X_v = a) &= 0,\\
\PP(v \text{ is neutral} \mid X_v = a) &= (1-q-\hat{q})^a\,\frac{a(1-q-\hat{q})}{\mu},\\  
\PP(v \text{ is negative} \mid X_v = a) &= 1-(1-q-\hat{q})^a\,\frac{a(1-q-\hat{q})}{\mu}.\\
\end{aligned}
\]   
It follows that
\[
\begin{aligned}
\PP(v \text{ is positive}) &= q+\hat{q} - \frac{1}{\mu}\big[q^3+2\hat{q}^4+3q^3\hat{q}+6q^2 \hat{q}^2+q\hat{q}^3\big],\\
\PP(v \text{ is neutral}) &= \frac{1}{\mu}\big[q^3 +2\hat{q}^4+ a (1-q-\hat{q})^{a+2}\big],\\  
\PP(v \text{ is negative}) &= (1-q-\hat{q}) + \frac{1}{\mu}\big[3q^3\hat{q}+6q^2\hat{q}^2+q\hat{q}^3-a\,(1-q-\hat{q})^{a+2}\big].
\end{aligned}
\]   
In the limit as $a\to\infty$,
\[
\begin{aligned}
f^+ =\PP(v \text{ is positive}) &\to q+\hat{q},\\
f^0 = \PP(v \text{ is neutral}) &\to 0,\\  
f^- = \PP(v \text{ is negative}) &\to 1-q-\hat{q}.
\end{aligned}
\]   
Therefore, from Definition \ref{def:infinite-significance} and Theorem \ref{thm:densities-on-infinite-trees}, for $q+\hat{q}\geq\tfrac12$ and $a$ large enough the friendship paradox for the tree is $\PP\text{-}\mathrm{a.s.}$ significant, while for $q+\hat{q}<\tfrac12$ and $a$ large enough the friendship paradox for the tree is $\PP\text{-}\mathrm{a.s.}$ insignificant. It is $\PP\text{-}\mathrm{a.s.}$  strictly significant for $q+\hat{q}>\tfrac12$ and $a$ large enough.
\end{proof}


\subsection{Correlations of vertex types in infinite Galton-Watson trees}
\label{sec:correlations-proof}

In this section we prove Theorems~\ref{thm:correlations-on-infinite-trees}--\ref{thm:plusplus-correlations-on-infinite-trees}. We use the same construction as in Section~\ref{ss:infGW}., and the proof goes along the same lines. 

\begin{proof}[Proof of Theorem \ref{thm:correlations-on-infinite-trees}]
We show that, for $\tchi ,\chi \in S$,
\begin{equation}
\label{eq:limit-tree*}
\frac{N^{\tchi\chi}_m}{N_m} \stackrel{\PP-\mathrm{a.s.}}{\longrightarrow} f^{\tchi\chi}, \qquad m\to\infty,
\end{equation}
with $f^{\tchi\chi}$ in Theorem \ref{thm:correlations-on-infinite-trees}(b). The proof follows the same line of thought as in the proof of Theorem~\ref{thm:densities-on-infinite-trees}. 

\medskip\noindent
{\bf 1.}
Fix $\tchi,\chi \in S$. For $v\in V(\Tt_\infty)$, let $Y_v^{\tchi\chi}$ be the number of $\tchi\chi$-edges $(vi,vij)\in E(\Tt_\infty)$ such that $vi$ has type $\tchi$ and its offspring has type $\chi$. We next derive an expression for $\EE\big[Y_v^{\tchi\chi}\big]$. Note that now the calculations span four generations instead of three, because the type of $vi$ depends on the degree of its parent $v$ and of its offspring $vij$. In turn, the type of vertex $vij$ depends on the degree of $vi$ and the degrees of the children of $vij$. We show that 
\begin{equation}\label{eq:tchichi}
\EE[Y_v^{\tchi\chi}]= \mu^2 f^{\tchi \chi}.
\end{equation}

\medskip\noindent
{\bf 2.}
By definition,
\begin{align*}
\EE[Y_v^{\tchi\chi}]
&= \EE\Big[\sum_{i=1}^{X_v} \ind{{\rm sign}\left[X_v+X_{vi1}+\cdots+X_{viX_{vi}}-X_{vi}(X_{vi}+1)\right]=\tchi}\\
&\quad\times\sum_{j=1}^{X_{vi}}\ind{{\rm sign}\left[X_{vi}+X_{vij1}+\cdots+X_{vijX_{vij}}-X_{vij}(X_{vij}+1)\right]=\chi}\Big]\\
&= \sum_{\tl\in\N_0} p_{\tl} \sum_{i=1}^{\tl} \EE\Big[\ind{{\rm sign}\left[\tl+X_{vi1}+\cdots+X_{viX_{vi}}-X_{vi}
(X_{vi}+1)\right]=\tchi} \\
&\quad\times\sum_{j=1}^{X_{vi}} \ind{{\rm sign}\left[X_{vi}+X_{vij1}+\cdots+X_{vijX_{vij}}-X_{vij}(X_{vij}+1)\right]=\chi}\Big],
\end{align*}
\noindent
where the first expression selects those children $i$ of $v$ for which a certain statistic has sign $\tchi$, and counts how many of their children $j$ satisfy a similar sign condition with respect to $\chi$, and in the second expression we condition on the root degree $X_v = \tl$. Since the $X_{vi}$ are i.i.d., we can choose a representative child, say $X_{v1}$, and multiply by $\tl$:
\begin{align*}
\EE[Y_v^{\tchi\chi}]&= \sum_{\tl\in\N_0} p_{\tl}\,\tl\,\EE\Big[\ind{{\rm sign}\left[\tl+X_{v11}+\cdots+X_{v1X_{v1}}-X_{v1}(X_{v1}+1)\right]=\tchi} \\
&\quad\times\sum_{j=1}^{X_{v1}} \ind{{\rm sign}\left[X_{v1}+X_{v1j1}
+\cdots+X_{v1jX_{v1j}}-X_{v1j}(X_{v1j}+1)\right]
=\chi}\Big].
\end{align*}

\medskip\noindent
{\bf 3.}
We next recall the size-biased distribution $\tilde{p}_{\tl} = \tl p_{\tl} / \mu$, and condition on the degree $X_{v1}=\tk$, to get
\begin{align*}
\EE[Y_v^{\tchi\chi}]
&= \mu \sum_{\tl\in\N_0} \tilde{p}_{\tl} \sum_{\tk\in\N_0} p_{\tk}\,\EE\Big[\ind{{\rm sign}\left[\tl+X_{v11}
+\cdots+X_{v1\tk}-\tk(\tk+1)\right]=\tchi} \\
&\quad\times\sum_{j=1}^{\tk}\ind{{\rm sign}\left[\tk+X_{v1j1}+\cdots+X_{v1jX_{v1j}}-X_{v1j}(X_{v1j}+1)\right]=\chi}\Big]\\
&= \mu \sum_{\tl\in\N_0} \tilde{p}_{\tl} \sum_{\tk\in\N_0} p_{\tk}\,\tk\,
\EE\Big[\ind{{\rm sign}\left[\tl+X_{v11}+\cdots+X_{v1\tk}-\tk(\tk+1)\right]=\tchi} \\
&\quad\times\ind{{\rm sign}\left[\tk+X_{v111}+\cdots+X_{v11X_{v11}}-X_{v11}(X_{v11}+1)\right]=\chi}\Big].
\end{align*}
In the last line we exploit exchangeability and fix a representative grandchild. We condition further on the value of $X_{v11} = k$ to separate the expectations and get that
\begin{align*}
\EE[Y_v^{\tchi\chi}]
& =\mu^2 \sum_{\tl\in\N_0} \tilde{p}_{\tl} \sum_{\tk\in\N_0} \tilde{p}_{\tk} \sum_{k\in\N_0} p_k
\EE\Big[\ind{{\rm sign}\left[\tl+k+X_{v12}+\cdots+X_{v1\tk}-\tk(\tk+1)\right]=\tchi} \\
&\quad\times\ind{{\rm sign}\left[\tk+X_{v111}+\cdots+X_{v11k}-k(k+1)\right]=\chi}\Big]\\
&= \mu^2 \sum_{\tk\in\N_0} \tp_{\tk} \sum_{k\in\N_0} p_k\,
\EE\Big[\ind{{\rm sign}\left[\tX+k+X_{v12}+\cdots+X_{v1\tk}-\tk(\tk+1)\right]=\tchi} \\
&\quad\times\ind{{\rm sign}\left[\tk+X_{v111}+\cdots+X_{v11k}-k(k+1)\right]=\chi}\Big],
\end{align*}
where we recall that $v111,\ldots,v11k$ denote the children of $v11$ under the Ulam-Harris labelling introduced in Section~\ref{sec:AP-not}. Above we identified the sum over $\tl$ as a random variable $\tX$ with law $\tilde{p}_{\tl}$, and used the law of total probability. Finally, using independence, we get
\begin{align*}
\EE[Y_v^{\tchi\chi}]
&= \mu^2 \sum_{\tk\in\N_0} \tp_{\tk} \sum_{k\in\N_0} p_k\,\PP\Big({\rm sign}\big[\tX+k+S_{\tk-1}-\tk(\tk+1)\big]=\tchi\Big) \\
&\quad\times\PP\Big( {\rm sign}\big[\tk+S_k-k(k+1)\big]=\chi\Big)=\mu^2 f^{\tchi\chi},
\end{align*}
which proves \eqref{eq:tchichi}.

\medskip\noindent
{\bf 4.}
The rest of the proof carries over almost verbatim as in the proof of Theorem~\ref{thm:densities-on-infinite-trees}, except that we need to consider two generations. Put
\[
Z_m^{\tchi\chi} = \sum_{v \in \partial \Tt^{m-2}_\infty} Y_v^{\tchi\chi}, 
\]
which counts the number of $\tchi\chi$-edges between generations $m-1$ and $m$. Since $Z_{m-2} \stackrel{\PP-\mathrm{a.s.}}{\longrightarrow} \infty$ as $m\to\infty$, by the similar argument as in the proof of Lemma~\ref{lem:randomindexSLLN} in the proof of Theorem~\ref{ss:infGW}, we get
\[
\frac{Z_m^{\tchi\chi}}{Z_{m-2}} \stackrel{\PP-\mathrm{a.s.}}{\longrightarrow} 
\mu^2 f^{\tchi\chi}, \qquad m\to \infty.
\]
Since, by \eqref{eq:KSthm*},
\[
\frac{Z_m}{Z_{m-2}} \stackrel{\PP-\mathrm{a.s.}}{\longrightarrow} \mu^2, \qquad m \to \infty,
\]
we get
\begin{align}
\label{eq:convergence-generation*}
\frac{Z_m^{\tchi\chi}}{Z_m} \stackrel{\PP-\mathrm{a.s.}}{\longrightarrow} f^{\tchi\chi}, \qquad m\to \infty.
\end{align}
The result in \eqref{eq:convergence-generation*} says that the fraction of $\tchi\chi$-edges between generations $m-1$ and $m$ converges almost surely to $f^{\tchi\chi}$ as $m\to\infty$. 
The same holds for the fraction of $\tchi\chi$-edges in the entire tree, by the same argument as we used for the density of vertices of type $\chi$. 
\end{proof}

To prove Theorem \ref{thm:plusplus-correlations-on-infinite-trees}, we need two auxiliary facts stated in the following lemma.
\begin{lemma} 
\label{l:aux}
Let $S_k$ and $\tilde{X}$ be as defined in Theorem \ref{thm:densities-on-infinite-trees}, and let $X$ be a random variable drawn independently from $p$. Subject to \eqref{eq:mono} the following hold:
\begin{itemize}
\item[(a)] 
For each $\tk$ in the support of $\tilde{p}$, the function 
\begin{equation}
\label{eq:mono>ap}
k \mapsto \PP\big(\tk + S_X - X(X+1) > 0 \mid X > k\big) \text{ is non-increasing on the support of } p.
\end{equation}
In particular, the function 
\begin{equation}
\label{eq:mono>}
k \mapsto \PP\big(\tX + S_X - X(X+1) > 0 \mid X > k\big) \text{ is non-increasing on the support of } p.
\end{equation}
\item[(b)] 
For any non-negative discrete random variable $Y$ that is independent of everything else, 
\begin{equation}
\label{eq:mono-st}
\PP(Y > k) \geq \PP\big(Y > k \mid \tX + S_Y - Y(Y+1) > 0\big),  \qquad k \in \N_0.
\end{equation}
\end{itemize}
\end{lemma}

\begin{proof}
(a) Abbreviate
\[
\rho(k)=\PP\big(\tk + S_k - k(k+1)>0\big).
\]
The condition in \eqref{eq:mono} says that, for each $\tk$ in the support of $\tilde{p}$, $k \mapsto \rho(k)$ is non-increasing on the support of $p$. Write
\begin{align*}
&\PP\big(\tk + S_X - X(X+1)>0\mid X>k\big) = \frac{1}{\sum_{l>k} p_l} \sum_{s>k} p_s\,\rho(s)\\
&\qquad = \left[\frac{p_{k+1}}{\sum_{l > k} p_l}\right]\,\rho(k+1)
+ \left[\frac{\sum_{l > k+1} p_l}{\sum_{l > k} p_l}\right] 
\,\frac{1}{\sum_{l > k+1} p_l}\,\sum_{s > k+1} p_s\,\rho(s)\\
&\qquad = \left[\frac{p_{k+1}}{\sum_{l > k} p_l}\right]\,\rho(k+1)
+ \left[1-\frac{p_{k+1}}{\sum_{l > k} p_l}\right] \,
\frac{1}{\sum_{l > k+1} p_l}\,\sum_{s > k+1} p_s\,\rho(s)\\
&\qquad \geq \frac{1}{\sum_{l > k+1} p_l}\,\sum_{s > k+1} p_s\,\rho(s) 
&\quad \tag{\scriptsize by \eqref{eq:mono}}\\
&\qquad = \PP\big(\tk + S_X - X(X+1)>0\mid X>k+1\big).
\end{align*}
(b) Abbreviate
\[
\rho(k)=\PP\big(\tX + S_k - k(k+1)>0\big).
\]
The condition in \eqref{eq:mono} says that $k \mapsto \rho(k)$ is non-increasing on the support of $p$. Write
\begin{align*}
&\PP\big(Y > k \mid \tX + S_Y - Y(Y+1) > 0\big)\\ 
&\qquad = \frac{\sum_{s>k} \PP(Y=s)\,\rho(s)}{\sum_{l \leq k} \PP(Y=l)\,\rho(l) 
+ \sum_{l> k} \PP(Y=l)\,\rho(l)}\\
&\qquad \leq \frac{\sum_{s>k} \PP(Y=s)\,\rho(s)}{\rho(k)\,\PP(Y\leq k)
+\sum_{l> k }\PP(Y=l)\,\rho(l)} \quad \tag{\scriptsize by \eqref{eq:mono}}\\
&\qquad \leq \frac{\rho(k+1)\,\PP(Y>k)}{\rho(k)\,\PP(Y\leq k)
+ \rho(k+1)\,\PP(Y>k)} \quad \tag{\scriptsize $x \mapsto \frac{x}{a+x}$ non-decreasing}\\
&\qquad \leq \frac{\rho(k+1)\,\PP(Y>k)}{\rho(k+1)\,\PP(Y \leq k) + \rho(k+1)\,\PP(Y>k)}
\quad \tag{\scriptsize by \eqref{eq:mono}}\\
&\qquad = \PP(Y>k).
\end{align*}
\end{proof}

\noindent
The intuition behind Lemma~\ref{l:aux} is as follows. Similarly as in \eqref{eq:mono}, the statement in \eqref{eq:mono>} says that bounding $X$ from below makes it harder for a vertex with a number of offspring $X$ to be positive.  The statement in \eqref{eq:mono-st} says that conditioning on the event that the vertex with the number of offspring $Y$ is positive makes $Y$ stochastically smaller. Note that \eqref{eq:mono-st} holds for any distribution $Y$. The offspring distribution $p$ participates through $S_Y = \sum_{i=1}^Y X_i$, just as in \eqref{eq:mono}. 

\begin{proof}[Proof of Theorem \ref{thm:plusplus-correlations-on-infinite-trees}]
From \eqref{AP-fxx}, using the notations introduced therein, we can write
\begin{align*}
f^{++} &= \PP\big(\tX^* + S_{\tX-1} + X - \tX(\tX+1) > 0,\,\tX + S_X^* -X(X+1) > 0\big)\\
&= \PP\big(\tX^* + S_{\tX-1} + X - \tX(\tX+1)>0 \big) \quad \tag{\scriptsize parent type +}\\
&\quad \times\,\PP\big(\tX + S_X^* - X(X+1) > 0 \mid \tX^* + S_{\tX-1} + X - \tX(\tX+1) > 0\big)
\quad \tag{\scriptsize child type + given parent type +}\\
&= \tf^+\,\PP\big(\tX + S_X^* - X(X+1) > 0 \mid \tX^* + S_{\tX-1} + X - \tX(\tX+1) > 0\big).
\end{align*}
It remains to show that the last conditional probability is $\leq f^{+}$. This comes in two steps. 

\medskip\noindent
{\bf 1.}
Estimate
\begin{align*}
&\PP\big(\tX + S_X^* - X(X+1) > 0 \mid \tX^* + S_{\tX -1}+X - \tX(\tX+1) > 0\big)\\
&\qquad = \sum_{\tk\in\N} \sum_{\tk^*\in\N} \sum_{s\in\N_0} 
\PP\big (\tX^*=\tk^*, \tX=\tk, S_{\tk-1}=s \mid \tX^* + S_{\tX} - \tX(\tX+1) > 0\big)\\
&\qquad \qquad \times \PP\big(\tk + S_X^* - X(X+1) > 0 \mid X  > \tk(\tk+1) - s - \tk^*\big)
\quad \tag{\scriptsize $S_{\tX-1}+X\stackrel{d}{=}S_{\tX}$}\\
&\qquad \leq \sum_{\tk\in\N} \sum_{\tk^*\in\N} \sum_{s\in\N_0} 
\PP\big (\tX^*=\tk^*, \tX=\tk, S_{\tk-1}=s \mid \tX^* + S_{\tX} - \tX(\tX+1) > 0\big)\\
&\qquad \qquad\times\PP\big(\tk + S_X^* -X(X+1) > 0)
\quad \tag{\scriptsize by \eqref{eq:mono>ap}}\\
&\qquad = \sum_{\tk\in\N_0} \PP\big (\tX=\tk \mid \tX^* + S_{\tX} - \tX(\tX+1) > 0\big)
\,\PP\big(\tk + S_X- X(X+1) > 0)
\quad \tag{\scriptsize sum over $\tk^*$ and $s$}\\
&\qquad = \PP\big(\tX^+ +S_X - X(X+1) > 0),
\end{align*}
where $\tX^+$ is the random variable with distribution $\PP(\tX^+=\tk) = \PP(\tX=\tk \mid \tX^* + S_{\tX} -\tX(\tX+1) > 0)$, $k\in\N_0$. 

\medskip\noindent
{\bf 2.} 
By \eqref{eq:mono-st} with $Y=\tX$, we have that $\tX$ is stochastically larger than $\tX^+$. Hence, from \eqref{eq:mono},
\[
\PP\big(\tX^+ + S_X - X(X+1) > 0) \leq \PP\big(\tX + S_X - X(X+1) > 0) = f^+.
\] 
\end{proof}

Intuitively, the negative correlation $f^{++} \leq \tf^+f^+$ holds for two reasons: (1) a child of a positive parent tends to have a large number of offspring, so that it is less likely to be positive; (2) a positive parent has a degree that tends to be small, so that its children are less likely to be positive. Each of these reasons should be enough to settle the negative correlation. However, our proof relies on both.


\subsection{Negative correlation in Galton-Watson trees}
\label{ap:Ex}

In this section we consider two classes of infinite Galton-Watson trees in which the correlation between the positivity of a parent and the positivity of its child is negative.

\begin{proof}[Proof of the statements in Example~\ref{AP-example}]
Let $S_k$ and $\tilde{X}$ be as defined in Theorem \ref{thm:densities-on-infinite-trees}. 

\medskip\noindent 
(a) Consider a Galton-Watson tree with offspring distribution $p_1\in (0,1)$, $p_a = 1-p_1$ for some $a \in \N\setminus\{1\}$, and $p_k = 0$ for all $k \notin \{1,a\}$. Then
\begin{align*}
f^{++}&=\tilde{p}_1\,p_1\,\PP\big(\tX >1\big)\,\PP\big(S_{1}>1\big)\\
&=\tilde{p}_1\,p_1\,\tilde{p}_a\,p_a <\tilde{p}_1\,p_1\,(\tilde{p}_a+p_a)^{2}.
\end{align*}
Moreover,
\begin{align*}
\tilde{f}^{+}
&=\tilde{p}_1\,\PP\big(\tX +S_{1} >2\big)+\tilde{p}_a\,\PP\big(\tX +S_{a} >a(a+1)\big) =\tilde{p}_1(\tilde{p}_a+p_a ),\\
f^{+}
&=p_1\,\PP\big(\tX +S_{1} >2\big)+p_a\,\PP\big(\tX +S_{a} >a(a+1)\big) =p_1(\tilde{p}_a+p_a ).
\end{align*}
Hence $f^{++}<\tilde{f}^{+}f^{+}$.

\medskip\noindent 
(b) Consider a Galton-Watson tree with offspring distribution $p_1\in (0,\frac{1}{2})$, $p_2\in (0,1-p_1)$, $p_a = 1-p_1-p_2\in (0,1)$ for some $a \in \N\setminus\{1,2,3,4\}$, and $p_k = 0$ for all $k \notin \{1,2,a\}$. Then
\begin{align*}
f^{++}
&=\tilde{p}_1\,p_1\,\PP\big(\tX >1\big)\,\PP\big(S_{1}>1\big)
+\tilde{p}_1\,p_2\,\PP\big(S_{2}>5\big)\\
&\quad +\tilde{p}_2\,p_1\,\PP\big(\tX +S_1 >5\big)
+\tilde{p}_2\,p_2\,\PP\big(\tX +S_1 >4\big)\,\PP\big(S_{2}>4\big) \\
&= \tilde{p}_1\,p_1(1-\tilde{p}_1)(1-p_1)+2\,\tilde{p}_1\,p_2\,p_{a}+\tilde{p}_2\,p_1(\tilde{p}_a+p_a) 
+2\,\tilde{p}_2\,p_2\,p_{a}(\tilde{p}_a+p_a).
\end{align*}
Also
\begin{align*}
\tilde{f}^{+}
&=\tilde{p}_1\,\PP\big(\tX +S_{1} >2\big)+\tilde{p}_2\,\PP\big(\tX +S_{2} >6\big) 
=\tilde{p}_1(2-\tilde{p}_1-p_1)+\tilde{p}_2(\tilde{p}_a+2p_a),\\
f^{+}
&=p_1\,\PP\big(\tX +S_{1} >2\big)+p_2\,\PP\big(\tX +S_{2} >6\big) 
=p_1(2-\tilde{p}_1-p_1)+p_2(\tilde{p}_a+2p_a).
\end{align*}
Hence 
\begin{align*}
\tilde{f}^{+}f^{+}&=\tilde{p}_1\,p_1(2-\tilde{p}_1-p_1)^{2}+\tilde{p}_1\,p_2 (2-\tilde{p}_1-p_1)(\tilde{p}_a+2p_a)\\
&\quad + \tilde{p}_2\,p_{1}(2-\tilde{p}_1-p_1)(\tilde{p}_a+2p_a)+\tilde{p}_2\,p_2 (\tilde{p}_a+2p_a)^2\\
&>
\tilde{p}_1\,p_1(1-\tilde{p}_1)(1-p_1)+\tilde{p}_1\,p_2 (2-\tilde{p}_1-p_1)(\tilde{p}_a+2p_a)\\
&\quad + \tilde{p}_2\,p_{1}(2-\tilde{p}_1-p_1)(\tilde{p}_a+2p_a)+2\,\tilde{p}_2\,p_2\,p_a (\tilde{p}_a+p_a).
\end{align*}
Since $\tilde{p}_1\leq p_1<\frac{1}{2}$, we have
\begin{align*}
&(2-\tilde{p}_1-p_1)(\tilde{p}_a+2p_a)> (\tilde{p}_a+2p_a) > 2p_a,\\
&(2-\tilde{p}_1-p_1)(\tilde{p}_a+2p_a)>(\tilde{p}_a+2p_a)>(\tilde{p}_a+p_a).
\end{align*}
Therefore $f^{++}<\tilde{f}^{+}f^{+}$.
\end{proof}



\end{document}